\documentclass[11pt]{amsart}
\usepackage[latin1]{inputenc}
\usepackage[english]{babel}
\usepackage{amsmath,amssymb,amsthm,amsfonts}

\usepackage{tikz-cd}

\makeatletter
\def\author@andify{
	\nxandlist {\unskip ,\penalty-1 \space\ignorespaces}
	{\unskip {} \@@and~}
	{\unskip \penalty-2 \space \@@and~}
}
\makeatother

\makeatletter
\@namedef{subjclassname@2010}{ \textup{2010} Mathematics Subject Classification}
\makeatother
\theoremstyle{definition}
\newtheorem{theorem}{Theorem}[section]
\theoremstyle{definition}
\newtheorem{lemma}[theorem]{Lemma}
\theoremstyle{definition}
\newtheorem{corollary}[theorem]{Corollary}
\theoremstyle{definition}
\newtheorem{proposition}[theorem]{Proposition}
\theoremstyle{definition}

\theoremstyle{definition}
\newtheorem{remark}[theorem]{Remark}
\theoremstyle{definition}
\newtheorem{example}[theorem]{Example}
\theoremstyle{definition}

\numberwithin{equation}{section}

\numberwithin{equation}{section}
\DeclareMathOperator{\supp}{supp}

\DeclareMathOperator*{\esssup}{esssup}

\DeclareMathOperator*{\interior}{int}

\newcommand{\norm}[1]{\left\lVert#1\right\rVert}
\newcommand{\vertiii}[1]{{\left\vert\kern-0.25ex\left\vert\kern-0.25ex\left\vert #1 
    \right\vert\kern-0.25ex\right\vert\kern-0.25ex\right\vert}}
\newcommand{\abs}[1]{\left\lvert#1\right\rvert}

\newcommand{\tandori}[1]{\widetilde{\mathit{#1}}}

\begin{document}
\title{Isomorphic and isometric structure of the optimal domains for Hardy-type operators}

\author{Tomasz Kiwerski*}
\address[Tomasz Kiwerski]{Faculty of Mathematics, Computer Science and Econometrics, University of Zielona G\'ora,
		prof. Z. Szafrana 4a, 65-516 Zielona G\'ora, Poland and \newline \indent
		Institute of Mathematics, Faculty of Electrical Engineering, Pozna\'{n} University of Technology,
		Piotrowo 3A, 60-965 Pozna\'{n}, Poland}
\email{tomasz.kiwerski@gmail.com}

\author{Pawe\l {} Kolwicz*}
\address[Pawe\l {} Kolwicz]{Institute of Mathematics, Faculty of Electrical Engineering, Pozna\'{n} University of Technology,
		Piotrowo 3A, 60-965 Pozna\'{n}, Poland}
\email{pawel.kolwicz@put.poznan.pl}

\author{Lech Maligranda}
\address[Lech Maligranda]{Department of Engineering Sciences and Mathematics, Luleå University of Technology,
		SE-971 87 Luleå, Sweden}
\email{lech.maligranda@ltu.se}

\begin{abstract}
We investigate structure of the optimal domains for the Hardy-type operators including, for example, 
the classical Ces\`aro, Copson and Volterra operators as well as for some of their generalizations. 
We prove that, in some sense, the abstract Ces\`aro and Copson function spaces are closely related 
to the space $L^1$, namely, they contain ``in the middle" a complemented copy of $L^1[0,1]$, 
asymptotically isometric copy of $\ell^1$ and also can be renormed to contain an isometric copy 
of $L^1[0,1]$. Moreover, the generalized Tandori function spaces are quite similar to $L^\infty$ 
because they contain an isometric copy of $\ell^\infty$ and can be renormed to contain an 
isometric copy of $L^\infty[0,1]$. Several applications to the metric fixed point theory will be 
given. Next, we prove that the Ces\`aro construction $X \mapsto CX$ does not commutate 
with the truncation operation of the measure space support. We also study whether a given 
property transfers between a Banach function space $X$ and the space $TX$, where $T$ 
is the Ces\`aro or the Copson operator. In particular, we find a large class of properties 
which do not lift from $TX$ into $X$ and prove that the abstract Ces\`aro and Copson 
function spaces are never reflexive, are not isomorphic to a dual space and do not have 
the Radon--Nikodym property in general.
\end{abstract}

\maketitle

\renewcommand{\thefootnote}{\fnsymbol{footnote}} \footnotetext[0]{
	{\it Date:} \today.  
		


	*The first and second authors are supported by the Ministry of Science and Higher Education of Poland, grant number 04/43/DSPB/0106.}

\section{Introduction}

In 1925 G. H. Hardy \cite{Ha25} proved the following inequality, which today is usually called the classical 
Hardy inequality
\begin{equation*}
\int_0^\infty \left(\frac{1}{x} \int_0^x f(t) \mathrm{d}t \right)^p \mathrm{d}x \leq \left(\frac{p}{p-1} \right)^p 
\int_0^\infty f(x)^p \mathrm{d}x,
\end{equation*}
where $1 < p < \infty$ and $f$ is a nonnegative real-valued Lebesgue measurable function 
(see \cite[Chapter 3]{KMP07} for more details). This inequality can be reformulated in the following way
\begin{equation*}
\text{the Hardy operator} \quad f \mapsto \frac{1}{x}\int_0^x f(t) \mathrm{d}t \quad \text{maps $L^p[0,\infty)$ 
continuously into itself}.
\end{equation*}

Given an operator $T \in \mathcal{L}(Y,X)$, where $X$ and $Y$ are Banach function spaces, it is natural to 
ask whether there is a Banach function space, say $Z$, such that $T \colon Z \rightarrow X$ is also bounded 
and $Z$ is the largest, in the sense of inclusion, Banach function space with this property. This situation 
can be summarized by the following diagram
\begin{center}
	\begin{tikzcd}
	Y \arrow[r, hook] \arrow[dr, "T"]
	& Z \arrow[d, dashrightarrow, "T"]\\
	& X
	\end{tikzcd}
\end{center}

Under some technical assumptions \cite[p. 196]{CR06}, $Z$ is the space of all measurable functions $f$ 
such that $T\abs{f} \in X$, equipped with the norm $\norm{f}_Z = \norm{T\abs{f}}_X$. In other words, 
the space $Z$ is the maximal or optimal domain for the operator $T$ considered with values in the 
fixed space $X$ and throughout this paper we adopt the convention to denote it by $TX$. This point 
of view turned out to be helpful and fruitful in the study of such classes of operators like kernel operators 
(special cases of operators in this class are, for example, the Volterra, Ces\`aro, Copson, Poisson 
or Riemann--Liouville operator), differential operators, convolutions, Fourier transform and the Sobolev 
embedding (see \cite{ORS08} and references given there).

The classical Ces\`aro and Copson function spaces appeared in a natural way as the optimal domains 
of the Hardy operator and its conjugate operator, respectively (see \cite{CR16}--\cite{DS07}, \cite{LM15a} 
and \cite{NP10}). For this reason and also to avoid the use of the term ``Hardy space", which is usually 
reserved for certain spaces of holomorphic functions on the unit disc (interestingly, introduced by 
F. Riesz in 1923 also to honour of Hardy), we will call mentioned operator the Ces\`aro operator 
$C$, i.e., $(Cf)(x) := \frac{1}{x}\int_0^x f(t) \mathrm{d}t$, remembering about his result on uniform 
convergence of averages of partial sums of Fourier series. There is also a connection between 
the Ces\`aro function spaces $CX$ and the so-called down spaces $X^\downarrow$ introduced 
by Sinnamon. Namely, for a symmetric space $X$ on $I = [0,\infty)$ such that the Ces\`aro operator 
$C$ is bounded on $X$ we can identify $CX$ with $X^\downarrow$ (see \cite{Si91}--\cite{Si07}; see 
also \cite{FLM16} and \cite[Section 3]{LM15a} for some additional remarks).

From an isomorphic point of view, the abstract Ces\`aro function spaces $CX$ as well as the abstract 
Copson function spaces $C^*X$ are a kind of a nontrivial mixture of the Banach function space $X$ 
and $L^1$ in which the properties of both spaces manifest themselves. Following this idea, we will 
look for ``the best possible" copies of the space $\ell^1$ and $L^1[0,1]$ in the Ces\`aro and Copson 
function spaces, and also of the space $\ell^\infty$ in $Ces_\infty := CL^\infty$.
We apply our results, among others, to the fixed point theory which is a very wide branch of 
functional analysis and has been developed for several decades (see \cite{gk} and \cite{KS01}). 
It has many applications, for example, in nonlinear analysis as well as in integral and differential 
equations. In particular, the question whether a Banach space $X$ has or fails the (weak) fixed point
property for nonexpansive mappings is a fundamental in this area.

In \cite[ Theorem 1 and 2]{AM08} Astashkin--Maligranda proved that the Ces\`aro function spaces 
$Ces_{p}:=CL^{p}$ for $1\leq p\leq \infty $ if $I=[0,1]$ and $1<p\leq \infty $ if $I=[0,\infty )$ fail to 
have the fixed point property for nonexpansive mappings. In contrast, it was proved by Cui--Hudzik 
\cite{CH99}, Cui--Hudzik--Li \cite{CHL00}, and Cui--Meng--P\l {}uciennik \cite{CMP00} that their 
sequence counterparts, i.e., the Ces\`aro sequence spaces $ces_{p}:=C\ell ^{p}$, have this 
property whenever $1<p<\infty $. We will show that the abstract Ces\`aro 
and Copson function spaces on two separable measure spaces $[0,1]$ and $[0,\infty )$ contain 
an order asymptotically isometric copy of $\ell ^{1}$ (the notions of asymptotic isometries are 
intermediate between the isomorphic and isometric theory) and thus, by the Dowling--Lennard 
result, fail to have the fixed point property in general. In the case of Ces\`aro function spaces this 
result can be seen as a generalization of the Astashkin--Maligranda result from \cite{AM08}. In
fact, the main idea to find an asymptotically isometric copy of $\ell ^{1}$ (which, by the way, were 
introduced precisely to show that certain spaces fail to have the fixed point property) remains 
the same but the argument is much more sophisticated and works in full generality. 
An analogous result for the Copson function spaces is new even for $X = L^{p}$. On the other 
hand, we also prove that nontrivial Tandori function spaces $\widetilde{\mathit{X}}$ contain an order
isomorphically isometric copy of $\ell ^{\infty }$ and consequently even fail to have the weak 
fixed point property. 

The second important problem we consider is the question whether ``some" property can be 
transfered from a simpler structure to more complicated one and vice versa. This type of problems 
has been successfully considered for many constructions. For example, we can mention three 
of such questions: $1^0$. $(X,E) \mapsto E(X)$, where $X$ is a Banach space, $E$ is a Banach 
function space and $E(X)$ is a K\"othe--Bochner space, $2^0$. $(X,Y) \mapsto \mathfrak{F}(X,Y)$, 
where $X$ and $Y$ are symmetric spaces and $\mathfrak{F}$ is an interpolation functor (see 
references in \cite{LM15p}), and $3^0$. $X \mapsto X^{(*)}$, where $X$ is a Banach function 
space and $X^{(*)}$ is the so-called symmetrization of $X$ (see references in \cite{KLM-2019}).

We will also consider this problem but for the Ces\`aro and Copson construction $X \mapsto TX$ 
presenting a large class of properties that never transfer from the space $TX$ into $X$.

Finally, we will examine the Ces\`aro construction $X \mapsto CX$ itself. More precisely, we show that this 
construction does not commutate in general with the truncation operation $X \mapsto X|_{[0,1]}$ 
highlighting in this way the difference between the Ces\`aro function spaces defined on $[0,1]$ 
and $[0,\infty)$. The question whether two operations commutate has been often investigated. 
For example, the symmetrization operation $X \mapsto X^{(*)}$ commutates with the 
Calder\'on--Lozanovski{\u \i} construction $\rho(X,Y)$ (in particular with the pointwise product $X \odot Y$)  
and with the pointwise multipliers $M(X,Y)$ (in particular with the K\"othe dual $X'$), see \cite{KLM-2019}. 
Furthermore, the Ces\`aro construction $X \mapsto CX$ comutates with the interpolation functor
$\mathfrak{F}$ having the homogenity property (see \cite[Theorem 6]{LM15p}).

It is worth to mention that we are able to prove most of the results without the assumption that the Ces\`aro 
or Copson operator is bounded on $X$, the assumption which is present in almost all results of this type. 

The paper is organized as follows.
After an introduction we collect some necessary definitions, basic facts and notations in Section 2. 
Here we also recall the duality theorem of Le\'snik--Maligranda from \cite{LM15a}, the 
Lindenstrauss--Tzafriri \cite{LT79} and Boyd \cite{Bo69} results on interpolation because we will 
use them frequently.

In Section 3 we will provide some basic results regarding nontriviality of the abstract Copson function 
spaces (Lemma \ref{Copson nontrivial} and Corollary \ref{Copson bounded and embeddings}). 
We discuss also the difference between the condition $TX \neq \{ 0 \}$ and the fact that the operator 
$T$ is bounded on $X$, where $T$ is the Ces\`aro or the Copson operator (Example \ref{Cesaro 
i Copson dla L1 i L8}). 

Section 4 starts with two lemmas (Lemma \ref{complementedL1} and Lemma \ref{complemented 
L1 in C*X}) which will play a crucial role later on. In particular, they show that the nontrivial Ces\`aro 
and Copson function spaces contain ``in the middle" a complemented copy of $L^1[0,1]$. Next, 
we prove that a Banach space $X$, which contains a complemented copy of a space $Y$, can 
always be renormed to contain an isometric copy of the space $Y$ (Theorem \ref{general 
renorming}). As a corollary we obtain immediately that the Ces\`aro and Copson function spaces 
can be renormed to contain an isometric copy of $L^1[0,1]$ and that the Tandori function spaces can be 
renormed to contain an isometric copy of $L^\infty[0,1]$. Finally, in Theorem \ref{AIC-l1} and 
Theorem \ref{AIC-l1 in C^*X} we present the main result of this section, namely, that the Ces\`aro 
and Copson function spaces always contain an order asymptotically isometric copy of $\ell^1$. 
Since the generalized Tandori function spaces $\widetilde{\mathit{X}}$ and the space $Ces_{\infty}$ 
are never order continuous \cite[Theorem 1 (e)]{LM15a} it follows that they contain an isomorphic 
copy of $\ell^\infty$. Nevertheless, we prove that $\widetilde{\mathit{X}}$ and $Ces_{\infty}$ 
even always contain an order isometric copy of $\ell^\infty$ (Proposition \ref{isometric l^infty in 
Tandori} and Proposition \ref{Ces_infty isometric copy of l8}). 

Next, in Section 5, we try to compare the Ces\`aro function spaces defined on $[0,1]$ and on 
$[0,\infty)$ and we show that the Ces\`aro construction $X \mapsto CX$, and the truncation operation 
of the measure space support $X \mapsto X|_{[0,1]}$ does not commutate in general 
(Lemma \ref{obciecie}). This fact explains, in a sense, quite supprising differences between some 
results obtained for the Ces\`aro function spaces on a finite and infinite interval (see also 
\cite{AM09}, \cite{LM15a} and \cite{LM15p}).

In Section 6, we analyze a problem of transfer of properties between $X$ and $TX$, where $T = C$ 
or $T = C^*$. We give an example of properties which lift from the Banach function space $X$ to 
$TX$ and vice versa (Corollary \ref{Copson a Cesaro} and Lemma \ref{X = Y, to CX = CY}). 
Next, using the results of Bessaga--Pe\l {}czy\'nski and Talagrand we obtain that the Ces\`aro and 
Copson function spaces are not isomorphic to a dual space and do not have the Radon--Nikodym 
property (Corollary \ref{not dual and not RNP}). Moreover, we include an example of a certain 
class of Banach function spaces which contain ``in the middle" an isomorphic copy of a Banach 
function space $Y$, but the construction $X \mapsto TX$, in a sense, forgets about this copy 
(Lemma \ref{CX zjada srodek}). The presented comparison of this example with the result from 
\cite{KT17} can be instructive. Furthermore, we  give a large class of properties (including, for 
example, order continuity, $p$-concavity and the Dunford--Pettis property) which do not transfer 
from $TX$ to $X$ (Theorem \ref{transfer}).  

The main result in Section 7 is Theorem \ref{main theorem}, which states that the abstract Ces\`aro 
and Copson function spaces fail to have the fixed point property in general. We prove that, under 
additional assumptions, these spaces cannot even be renormed to have the fixed point property
(Corollary \ref{TX bez przenormowania do FPP}). We also conclude that the generalized Tandori 
function spaces $\widetilde{\mathit{X}}$ and the space $Ces_{\infty}$ fail to have the weak fixed 
point property (Proposition \ref{Proposition wFPP}). 

Section 8 presents a certain way to generalize the results from the previous sections. We show that 
the methods developed by us in Sections 4 and 7 also work for a wider class of operators, e.g., 
for the weighted Ces\`aro operator $\mathcal{H}_w$ and its conjugate $\mathcal{H}_w^{*}$ 
(Theorem \ref{AIC-l1 uogolnione C i C*} and Theorem \ref{uogolnione Ci C* nie maja FPP}). 
In particular, we prove that an abstract Volterra spaces $VX$ fail to have the fixed point property 
as well (Corollary \ref{VX nie ma FPP}).

Finally, the Appendix is devoted to the analysis of a certain objects, specifically, two functions $F_X$ 
and $G_X$, that appeared in the proof of Theorems \ref{AIC-l1} and \ref{AIC-l1 in C^*X}. We finish 
this section with a few examples (Example \ref{przyklad do F_X}). In the first one we give some 
rather exotic examples of the function $F_X$ and in the next we will justify that the order continuity 
of a symmetric space $X$ is not crucial for the continuity of the function $F_X$.

\section{Notation and preliminaries}

\subsection{Banach function spaces and symmetric spaces}

Denote by $m$ the Lebesgue measure on $I$, where $I = [0,1]$ or $I = [0, \infty)$, and by $L^0 = L^0(I)$ 
the set of all equivalence classes of real-valued Lebesgue measurable functions defined on $I$.
A \textit{Banach function space} (or a \textit{Banach ideal space}) $X = (X, \norm{\cdot}_X)$ on $I$
is understood to be a Banach space $X$ such that $X$ is a linear subspace of $L^0(I)$ satisfying 
the so-called \textit{ideal property}, which means that if $f,g \in L^0(I)$, $\abs{f(t)} \leq \abs{g(t)}$ 
for almost all $t \in I$ and $g \in X$, then $f \in X$ and $\norm{f}_X \leq \norm{g}_X$.
If it is not stated otherwise we assume that a Banach function space $X$ contains a function 
$f_0 \in X$ which is positive almost everywhere (in short, {\it a.e.}) on $I$ (such a function is called the 
\textit{weak unit} in $X$), which means that $\supp(X) = I$. Sometimes we will write $X[0,1]$ or 
$X[0,\infty)$ to clearly indicate that a Banach function space $X$ is defined on $I = [0,1]$ or on 
$I = [0,\infty)$, respectively. We say that a Banach function space $X$ is \textit{nontrivial} if 
$X \neq \left\{ 0\right\}$.

For two Banach function spaces $X$ and $Y$ on $I$, the symbol $X \lhook\joinrel\xrightarrow{M} Y$ 
denotes the fact that the inclusion $X \subset Y$ is continuous with the norm not bigger than $M$, i.e., 
there exists a constant $M > 0$ (we will call it the embedding constant) such that 
$\norm{f}_Y \leq M\norm{f}_X$ for all $f \in X$. If the embedding $X \lhook\joinrel\xrightarrow{M} Y$ 
holds with some (maybe unknown) constant $M > 0$ we simply write $X \hookrightarrow Y$ and 
$\norm{f}_Y \lesssim \norm{f}_X$. Recall also that for two Banach function spaces $X$ and $Y$ 
the inclusion $X \subset Y$ is always continuous.
Moreover, $X = Y$ (resp. $X \equiv Y$) means that the spaces $X$ and $Y$ have the same 
elements and their norms are equivalent (resp. equal). If the spaces $X$ and $Y$ are isomorphic 
(resp. are isometric under the isometry $\lambda \cdot \text{id}$, where $\lambda > 0$), then 
we write $X \simeq Y$ (resp. $X \cong Y$). 

Let us remind that the {\it K\"othe dual space} (or {\it associated space}) $X' = X'(I)$ of a Banach 
function space $X$ on $I$ is defined as 
$$ X' := \{ f \in L^0(I) \colon \norm{f}_{X'} = \sup\limits_{g\in X, \ \norm{g}_X \leq 1} 
\int_I \abs{f(x)g(x)}\mathrm{d}x < \infty \}.$$
The K\"othe dual space is again a Banach function space. Moreover, $X \lhook\joinrel\xrightarrow{1} X'' := (X')'$ 
and $X = X''$ if and only if the norm in $X$ has the {\it Fatou property} (in short $X \in (FP)$), i.e., 
if for any sequence $(f_n) \subset X$ with $0 < f_n \uparrow f$ almost everywhere on $I$ such that 
$\sup\limits_{n\in \mathbb{N}} \norm{f_n}_X < \infty$, we have $f \in X$ and $\norm{f_n}_X \uparrow \norm{f}_X$.

A function $f \in X$, where $X$ is a Banach function space space on $I$, is said to have an {\it order 
continuous norm} in $X$ if for any decreasing sequence of sets $A_n \subset I$ with empty intersection, 
we have $\norm{f\chi_{A_n}}_X \rightarrow 0$ as $n\rightarrow \infty$ (see \cite[Proposition 3.5, p. 15]{BS88}).
By $X_a$ we denote the {\it subspace of all functions with order continuous norm in} $X$.
A Banach function space space $X$ on $I$ is said to be \textit{order continuous} (we write $X \in (OC)$ 
for short) if every element of $X$ has an order continuous norm, that is, if $X_a = X$. The subspace $X_a$ 
is always closed in $X$ (cf. \cite[Th. 3.8, p. 16]{BS88}). If $X$ is an order continuous Banach function 
space then $X^* = X'$ (see \cite[Theorem 4.1, p. 20]{BS88}). Moreover, a Banach function space on 
$I$ with the Fatou property is reflexive if and only if both $X$ and $X'$ are order continuous 
(cf. \cite[Corollary 4.4, p. 23]{BS88}).

Throughout the paper, we will accept the convention that whenever we take a subset $A \subset I$, 
we mean that $A$ is a Lebesgue measurable set. For a function $f \in L^0(I)$ we define the 
\textit{support of} $f$ as $$\supp(f) := \{x \in I \colon f(x) \neq 0 \}.$$

For a measurable function $w : I \rightarrow (0,\infty)$ (the \textit{weight on} $I$) and for a Banach function 
space $X$ on $I$, the \textit{weighted Banach function space} $X(w) = X(w)(I)$ is defined as
\begin{equation*}
X(w) := \{ f \in L^0(I) : fw \in X \},
\end{equation*}
with the norm $\norm{f}_{X(w)} = \norm{fw}_X$. It is clear that $X(w)$ is a Banach function space on $I$ 
and $X(w)' \equiv X'(1/w)$.

For a function $f \in L^0(I)$ we define the \textit{distribution function} $d_f(\lambda) := m(\{  t \in I : \abs{f(t)} > \lambda \})$ 
for $\lambda > 0$. We say that two functions $f, g \in L^0(I)$ are \textit{equimeasurable} when they have 
the same distribution functions, i.e. $d_f \equiv d_g$. By a \textit{symmetric space} (\textit{symmetric Banach 
function space} or \textit{rearrangement invariant Banach function space}) on $I$ we mean a Banach function 
space $E = (E, \norm{\cdot}_E)$ on $I$ with the additional property that for any two equimeasurable 
functions $f,g \in L^0(I)$ if $f \in E$ then $g \in E$ and $\norm{f}_E = \norm{g}_E$.
In particular, $\norm{f}_E = \norm{f^*}_E$, where $f^*(t) := \inf\{\lambda > 0 : d_f(\lambda) \leq t \}$ 
for $t \geq 0$.

For general properties of Banach lattices, Banach function spaces and symmetric spaces we refer 
to the books by Bennett--Sharpley \cite{BS88}, Kantorovich--Akilov \cite{KA82},
Krein--Petunin--Semenov \cite{KPS82}, Lindenstrauss--Tzafriri \cite{LT79}, Maligranda \cite{Ma89},
Meyer-Nieberg \cite{MN91}, and Wnuk \cite{Wn99}.

\subsection{Ces\`aro, Copson and Tandori function spaces} \label{subsekcja CX, C*X i Tandori}

For a Banach function space $X$ on $I$ the \textit{abstract Ces\`aro function space} $CX = CX(I)$ 
is defined as 
$$
CX := \{ f \in L^0(I) \colon C\abs{f} \in X \} \quad \text{with the norm} \quad \norm{f}_{CX} := \norm{C\abs{f}}_X,
$$
where $C$ denotes the \textit{Ces\`aro operator} (sometimes also called the Hardy operator)
$$
C \colon f \mapsto Cf(x) := \frac{1}{x} \int_0^x f(t) \mathrm{d}t \quad \text{for} \quad 0 < x \in I.
$$
The Copson and Tandori spaces are directly related to the Ces\`aro spaces. For a Banach ideal space 
$X$ on $I$ we define the \textit{abstract Copson function space} $C^*X = C^*X(I)$ as
$$
C^*X := \{ f \in L^0(I) \colon C^*\abs{f} \in X \} \quad \text{with the norm} \quad \norm{f}_{C^*X} := \norm{C^*\abs{f}}_X,
$$
where $C^*$ denotes the conjugate operator (in the sense of K\"othe) to the Ces\`aro operator $C$, which will be called
the \textit{Copson operator}, that is
$$
C^* \colon f \mapsto C^*f(x) := \int_{I\cap [x,\infty)} \frac{f(t)}{t} \mathrm{d}t \quad \text{for} \quad x \in I,
$$
and the \textit{abstract Tandori function space} $\widetilde{\mathit{X}} = \widetilde{\mathit{X}}(I)$ as
$$ 
\tandori{X} := \{ f \in L^0(I) \colon \tandori{f} \in X \} \quad \text{with the norm} \quad \norm{f}_{\tandori{X}} := \norm{\tandori{f}}_X,
$$
where by the \textit{nonincreasing majorant} $\tandori{f}$ of a given function $f$ we understand by
$$
\tandori{f}(x) := \esssup\limits_{t\in I, t \geq x} \abs{f(t)} \quad \text{for} \quad x \in I.
$$

The abstract Ces\`aro function spaces are simply a generalization of the well-known classical Ces\`aro spaces 
$Ces_p[0,1]$ and $Ces_p[0,\infty)$. Indeed, if we take $X = L^p$, where $1 \leq p \leq \infty$, then 
$Ces_p = CL^p$ (note, that in the case when $p = 1$ we have $Ces_1[0,1] = L^1(\ln(1/t))$ and 
$Ces_1[0,\infty) = \{0 \} $). The space $Ces_\infty[0,1]$ appeared already in 1948 and it is known 
as the Korenblyum--Kre\u \i n--Levin space $K$ (see \cite{KKL48}, \cite[p. 26 and p. 61]{Wn99} 
and \cite[pp. 469--471]{Za83}).

Various properties of these spaces have been studied by Astashkin in \cite{As12}, Astashkin--Maligranda 
in \cite{AM08}--\cite{AM14}, Hassard--Hussein in \cite{HH73}, Kami\'{n}ska--Kubiak in \cite{KK12}, Kubiak 
in \cite{Ku14}, Shiue in \cite{Sh70} and Sy--Zhang--Lee in \cite{SZL87}. Taking $X = L^\Phi$, 
$X = \Lambda_\varphi$ or $X = M_\varphi$ we obtain the Ces\`aro--Orlicz, Ces\`aro--Lorentz and 
Ces\`aro--Marcinkiewicz spaces, respectively, which have been studied intensively by 
Astashkin--Le\'{s}nik--Maligranda in \cite{ALM17}, Kiwerski--Kolwicz in \cite{KK16}--\cite{Ki-Kol-isom}, 
and Kiwerski--Tomaszewski in \cite{KT17}. 
General consideration of this construction when $X$ is a Banach function space or sometimes a symmetric 
space were initiated in \cite{LM15a} and \cite{LM15b}. More recently, the structure of these spaces, 
especially in their general form, is quite popular to study among various researchers such as 
Astashkin--Le\'snik--Maligranda \cite{ALM17}, Curbera--Ricker \cite{CR16}, Delgado--Soria \cite{DS07},
and Kiwerski--Tomaszewski \cite{KT17}.

Note that Ces\`aro function spaces $CX$ are never symmetric nor reflexive. Nevertheless, at least when 
$X$ is a symmetric space, there are some connections and similarities with the classical theory of normed 
ideal spaces and symmetric spaces. For example, it has been shown in \cite[Theorem 3]{KT17}
that order continuity property ``transfers" quite well between $X$ and $CX$. Moreover, $Ces_\infty$ 
and $ces_\infty$ are isomorphic, see \cite[Theorem 13]{ALM17} (this is analogous to the well-known 
Pe\l czy\'{n}ski result \cite{Pe58}, which states that spaces $L^\infty$ and $\ell^\infty$ are isomorphic).
Furthermore, $\widetilde{\mathit{\ell^1}}$ has the Schur property but is not isomorphic to $\ell^1$ \cite[Theorem 3.1]{ALM17}. 
Of course, there are also big differences if we compare the results obtained in the cases of a finite 
and infinite interval. For example, this differences can be seen in results on the K\"{o}the duality for 
abstract Ces\`aro function spaces in \cite[Theorems 3, 4 and 5]{LM15a} (cf. also Theorem A below) 
or in the interpolation results proved in \cite{LM15p}.

It is worth mentioning here that the study of the classical Ces\`aro sequence spaces $ces_p = C\ell^p$ 
for $1 < p \leq \infty$ began much earlier and many results have been obtained, see \cite{AM09} 
and \cite{AM14} and the references therein.

Copson function spaces $Cop_p = C^*L^p$ and Copson sequence spaces $cop_p = C^*\ell^p$ have 
appeared already in Bennett's memoir \cite[pp. 25--28 and p. 123]{Be96}. Furthermore, 
Astashkin--Maligranda used Copson function spaces $Cop_p$ to describe their interpolation results, 
see \cite[Section 2]{AM13a}. The abstract Copson spaces have been studied by Le\'{s}nik--Maligranda 
in \cite{LM15p}. For some connections between the Ces\`aro and Copson function spaces
and even their iterations $CCX$ and $C^*C^*X$ we refer to \cite[Theorem 1 (a) and (b)]{LM15p}.

Le\'snik and Maligranda suggested in \cite{LM15a} to call $\tandori{X}$ the generalized Tandori spaces 
since Tandori \cite{Ta55} proved in 1954 that $(Ces_\infty[0,1])' = \tandori{L^1}[0,1]$. Moreover, these 
spaces appeared earlier but without such name, e.g., in \cite{AM09} and \cite{LM15b}.
The Tandori spaces are related to the K\"{o}the duality of Ces\`aro spaces. Many special cases of 
this general construction have been studied by Alexiewicz \cite{Al57}, Astashkin--Maligranda \cite{AM09}, 
Bennett \cite{Be96}, Jagers \cite{Ja74}, Kami\'{n}ska--Kubiak \cite{KK12} and Luxemburg--Zaanen \cite{LZ66}.
The general Tandori spaces $\tandori{X}$ have been studied by Le\'{s}nik--Maligranda in 
\cite{LM15a}--\cite{LM15p} and the following K\"othe duality result was proved in \cite[Theorems 3, 5 and 6]{LM15a}.

\bigskip \noindent
{\bf Theorem A}. {\it If $X$ is a Banach function space on $I = [0,\infty)$ such that
the Ces\`aro operator $C$ and the dilation operator $\sigma_\tau$ (for some $0 < \tau < 1$) are 
bounded on $X$, then
\begin{equation} \label{Theorem A na [0,infty]}
(CX)' = \tandori{X'}.
\end{equation}
Furthermore, if $X$ is a symmetric space on $I = [0,1]$ with the Fatou property such that both 
operators $C$ and $C^*$ are bounded on $X$, then}
\begin{equation} \label{Theorem A na [0,1]}
(CX)' = \tandori{X'(w)} \quad \text{where} \quad w \colon [0,1) \ni x \mapsto \frac{1}{1-x}.
\end{equation}
\smallskip

The {\it dilation operator} $\sigma_\tau$ for $\tau > 0$ is defined by $\sigma_\tau f(x) := f(x/\tau)$ for 
$0 < x < \infty$ and
\begin{equation*}
\sigma_\tau f(x) := \begin{cases} f(x/\tau), & \mbox{if } x < \min\{1,\tau \} \\
									0 & \mbox{if } \tau \leq x < 1 \end{cases},
\end{equation*}
for $0 < x \leq 1$. This operator is bounded in any symmetric space $X$ on $I$ and 
$\norm{\sigma_\tau}_{X\rightarrow X} \leq \max\{1, \tau\}$ (see \cite[p. 148]{BS88} and \cite[pp. 96--98]{KPS82}).
The \textit{Boyd indices} of a symmetric space $X$ are defined by
\begin{equation*}
p(X) := \lim\limits_{\tau \rightarrow \infty} \frac{\ln \tau}{\ln \norm{\sigma_\tau}_{X \rightarrow X}} \quad \text{and} \quad
q(X) := \lim\limits_{\tau \rightarrow 0^+} \frac{\ln \tau}{\ln \norm{\sigma_\tau}_{X \rightarrow X}}.
\end{equation*}
Let us mention that these numbers can be different for $X$ on $I = [0, 1]$ and for $X$ on $I = [0, \infty)$, but
always we have estimates $1 \leq p(X) \leq q(X) \leq \infty$ (see \cite{KPS82}, \cite{LT79} and \cite{Ma85}).
\bigskip

We will use the following result from the Lindenstrauss--Tzafriri book \cite[Proposition 2.b.3, p. 132]{LT79}.

\bigskip \noindent
{\bf Theorem B}. {\it If $X$ is a symmetric space on $I$, then there are constants $A, B > 0$ such that
\begin{equation} \label{wlozenie}
L^{p}\cap L^{q} \lhook\joinrel\xrightarrow{A} X \lhook\joinrel\xrightarrow{B} L^{p}+L^{q},
\end{equation}
for every $p, q > 0$ satisfying $1 \leq p < p(X)$ and $q(X) < q \leq \infty$, where $p(X)$ and $q(X)$ 
are the Boyd indices of the space $X$,
$$L^{p}\cap L^{q} := \{f \in L^0(I) : \norm{f}_{L^p \cap L^q} = \max\{\norm{f}_{L^p}, \norm{f}_{L^q} \} < \infty \},$$
and
$$L^{p}+L^{q} := \{f \in L^0(I) : \norm{f}_{L^{p}+L^{q}} = \inf\limits_{\substack{f = g + h \\ g \in L^p, \ h \in L^q}} \{ \norm{g}_{L^p} + \norm{h}_{L^q}\} < \infty \}.$$
Moreover, if $p(X) = 1$ (resp. $q(X) = \infty$) then we can take $p = 1$ (resp. $q = \infty$) in (\ref{wlozenie}).}
\bigskip

Let us recall the important result about boundedness of the Ces\`aro operator (cf. \cite[Theorem 17, p. 130]{KMP07}).

\bigskip \noindent
{\bf Theorem C}. {\it Let $X$ be a symmetric space on $I$. Then
\begin{enumerate}
\item [(i)] the Ces\`aro operator $C$ is bounded on $X$ if and only if $p(X) > 1$,
\item [(ii)] the Copson operator $C^*$ is bounded on $X$ if and only if $q(X) < \infty$.
\end{enumerate}
}
\bigskip

Throughout the article we will use the following notation: the norm of the function
$f_{\lambda } \colon I\ni x\longmapsto \frac{1}{x}\chi_{\lbrack \lambda, m(I))}(x)$, where $0 < \lambda \in I$,
in a Banach function space $X$ on $I$, will be denoted by
$\left\Vert \frac{1}{x}\chi _{\lbrack \lambda, m(I))}(x)\right\Vert _{X(I)}$,
i.e.,
\begin{equation*}
\left\Vert f_{\lambda }\right\Vert _{X(I)}:=\left\Vert \frac{1}{x}\chi
_{\lbrack \lambda ,m(I))}(x)\right\Vert _{X(I)},
\end{equation*}
and the norm of the function $(C\left\vert f\right\vert)\chi _{A} \colon I\ni x\longmapsto 
\frac{1}{x}\int_{0}^{x}\left\vert f(t)\right\vert \mathrm{d}t\chi _{A}\left(
x\right)$, where $A\subset I$, will
be denoted by
\begin{equation*}
\left\Vert \frac{1}{x}\int_{0}^{x}\left\vert f(t)\right\vert \mathrm{d}t\chi
_{A}\left( x\right) \right\Vert _{X(I)}:=\left\Vert \left(
C\left\vert f\right\vert \right) \chi _{A}\right\Vert _{X(I)}.
\end{equation*}

Recall that if $X$ is a Banach function space on $I$ and $X \in (FP)$, then the Ces\`aro operator $C$ 
is bounded on $X$ if and only if the Copson operator $C^*$ is bounded on $X'$ and 
$\norm{C}_{X \rightarrow X} = \norm{C^*}_{X' \rightarrow X'}$ (see \cite[Remark 1 (iv)]{KLM-2019}).
Note also that if $X$ is a Banach function space on $I$, then the assumption $C \colon X \rightarrow X$ 
is in fact equivalent to the statement that the Ces\`aro operator $C$ is bounded on $X$ (see \cite{KT17}). 
Clearly, if the operator $C$ is bounded on $X$, then $X \hookrightarrow CX$. Therefore, the space 
$CX$ is nontrivial with $\supp(CX) = \supp(X) = I$. We will now collect some other useful facts about 
an abstract Ces\`aro function space $CX$, which are proved in \cite[the proof of Proposition 2.2]{ALM17}, 
\cite[Theorem 1 (a) and (b)]{LM15a} and \cite[Lemma 2]{KT17}.

\bigskip \noindent
{\bf Theorem D}. {\it Let $X$ be a Banach function space on $I$. Then
\begin{enumerate}
\item [(i)] $CX[0,1]$ is nontrivial if and only if $\chi_{[\lambda, 1]} \in X$ for some $0 < \lambda < 1$.
\item [(ii)] $CX[0,\infty)$ is nontrivial if and only if $\frac{1}{x}\chi_{[\lambda, \infty)}(x) \in X$ for some $\lambda > 0$.
\end{enumerate}
In particular, $[\lambda, m(I)) \subset \supp(CX)$ for some $0 < \lambda < m(I)$.
	
If $X$ is a Banach function space on $I$ such that the Ces\`aro operator $C$ is bounded on $X$ or $X$ 
is a symmetric space on $[0,1]$ or $X$ is a symmetric space on $[0,\infty)$ with $CX[0,\infty) \neq \{ 0 \}$, then
\begin{enumerate}
\item [(iii)] $\chi_{[\lambda, 1]} \in X$ for all $0 < \lambda < 1$ if $I = [0,1]$,
\item [(iv)] $\frac{1}{x}\chi_{[\lambda, \infty)}(x) \in X$ for all $\lambda > 0$ if $I = [0,\infty)$.
\end{enumerate}
In particular, $\supp(CX) = \supp(X) = I$. Let us emphasize also that if $X$ is a symmetric space on $[0,1]$, 
then $CX$ is always nontrivial.
}
\bigskip

\section{Some auxiliary results}

We will give below a few simple but useful facts about the Copson spaces.

\begin{lemma} \label{Copson nontrivial}
{\it Let $X$ be a Banach function space on $I$. Then the Copson space $C^*X$ is nontrivial if 
and only if $\chi_{[0,\lambda]} \in X$ for some $0 < \lambda < m(I)$.}
\end{lemma}
\begin{proof}
Assume that $C^*X \neq \{ 0 \}$. Then there exists $f \in C^*X$ with $\abs{f(x)} > 0$ for 
$x \in A \subset I$ and $m(A) > 0$. Of course, we can also find $\lambda > 0$ such that
$$ 
\int_\lambda^{m(I)} \frac{\abs{f(t)}}{t} \mathrm{d}t := \eta > 0. 
$$
Therefore, we obtain that
\begin{align*}
\eta \chi_{[0,\lambda]}(x) & = \int_\lambda^{m(I)} \frac{\abs{f(t)}}{t} \mathrm{d}t \chi_{[0,\lambda]}(x) \\
& \leq \int_x^{m(I)} \frac{\abs{f(t)}}{t} \mathrm{d}t \chi_{[0, \lambda]}(x)
\leq C^*\abs{f}(x) \in X,
\end{align*}
so $\chi_{[0,\lambda]} \in X$.
	
If $\chi_{[0,\lambda]} \in X$ for some $0 < \lambda < m(I)$, then for $0 < a < \lambda$ we have
$$ 
\norm{\chi_{[a, \lambda]}}_{C^*X} \leq \norm{(\int_a^\lambda \frac{\mathrm{d}t}{t}) \chi_{[0,\lambda]}}_X
= \ln(\frac{\lambda}{a}) \norm{\chi_{[0,\lambda]}}_X < \infty,
$$
which means that $C^*X \neq \{ 0 \}$.
\end{proof}

\begin{corollary} \label{Copson bounded and embeddings}
(1) {\it The Copson space $C^*X$ is always nontrivial whenever $X$ is a symmetric space.}

(2) {\it If $X$ is a Banach function space on $I$ such that the operator $C^*$ is bounded on 
$X$, then $\supp(C^*X)=I$. In particular, the Copson space $C^*X$ is nontrivial. Moreover,
\begin{enumerate}
\item [(i)] $L^\infty[0,1]|_{[0,\lambda]} \hookrightarrow X[0,1]$ for all $0 < \lambda < 1$
		and, in addition, $L^\infty[0, 1] \hookrightarrow X[0,1]$ if $X$ has the Fatou property,
\item [(ii)] $L^\infty_{\text{fin}}[0,\infty) \subset X[0, \infty)$
		and, in addition, $L^\infty_b[0,\infty) \hookrightarrow X[0, \infty)$ if $X$ has the Fatou property,
\end{enumerate}
where $L^\infty_{\text{fin}}(I) := \{f \in L^\infty(I) \colon m(\supp(f)) < \infty \}$ and $(L^\infty(I))_b = L^\infty(I)_b$
is the closure of $L^\infty_{\text{fin}}(I)$} in $L^\infty(I)$.
\end{corollary}
\begin{proof}
If $X$ is a symmetric space on $I$, then $\chi_{[0,\lambda]} \in X$ for all $0 < \lambda < m(I)$,
so $C^*X \neq \{ 0 \}$, see Lemma \ref{Copson nontrivial}.
	
It is also clear that if the operator $C^*$ is bounded on $X$ then $X \hookrightarrow C^*X$ and
consequently $\supp(C^*X)=I$.
	
(i) Of course, the condition $L^\infty[0,1]|_{[0,\lambda]} \hookrightarrow X$ is equivalent to 
$\chi_{[0,\lambda]} \in X$. Take $0 < \lambda < 1$ and let $f_0$ be a weak unit in $X$. Then 
$\int_{\lambda}^{m(I)} \abs{f_0(t)}/t \mathrm{d}t := \delta > 0$ and proceeding as in the first part 
of the proof of Lemma \ref{Copson nontrivial} we get $\chi_{[0,\lambda]} \in X$. 
If, additionally, $X \in (FP)$, then we conclude that $\chi_{[0,1]} \in X$, i.e., $L^\infty[0,1] \hookrightarrow X$.
	
(ii) Similarly, as in the case (i), above we obtain that $L^\infty_{\text{fin}}[0,\infty) \subset X[0, \infty)$ 
(note only that $L^\infty_{\text{fin}}$ is not complete whence the inclusion $L^\infty_{\text{fin}} \subset L^\infty_b$ 
is not continuous). Suppose now that $X \in (FP)$ and take $f \in L^\infty_b[0,\infty)$. To show that 
$f \in X$ it is enough to take a sequence $(f_n) \subset L^\infty_\text{fin}[0,\infty)$ with $0 \leq f_n \uparrow f$.
\end{proof}

Many results, although we can probably say that almost all, in the theory of Ces\`aro and Copson function 
spaces are proved under the assumption that at least one of the operators $C$ or $C^{\ast}$ is bounded 
on $X$. As mentioned in the introduction, we are able to prove our results under the essentially weaker 
assumption (actually the weakest possible one) that the Ces\`aro or Copson function space is nontrivial.
In this context, it seems resonable to give several examples discussing the difference between these 
two assumptions, because many naturally appearing spaces have the property that the operator $T$, 
where $T = C$ or $T = C^*$, is not bounded on $X$ but $TX \neq \{ 0 \}$ or even $\supp(TX) = I$.

\begin{example} \label{Cesaro i Copson dla L1 i L8}
(a) Consider, as in \cite[Example 2]{LM15a}, the space $L^p(w_1)$ on $[0,\infty)$, where $1 < p < \infty$ and
$$
w_1(x) = \frac{1}{1-x}\chi_{[0,1)}(x) + \chi_{[1,\infty)}(x).
$$
Then $\supp(X) = [0,\infty)$, $\supp(CX) = [1,\infty)$ and $\supp(C^*X) = [0,1]$.
Put $X = L^1(w_2)$, where
$$
w_2 \colon I \ni x \mapsto \frac{1}{x}.
$$
Then it is clear that $\chi_{[0,\lambda]} \notin X$ for every $0 < \lambda < m(I)$, so  $C^*X = \{ 0 \}$.
Finally, take $X = L^\infty(w_3)$, where
$$
w_3 = \text{id}_I \colon I \ni x \mapsto x.
$$
Then $CX\equiv L^{1}$. On the other hand, if $C \colon X \rightarrow X$ is bounded, then 
$X\hookrightarrow CX$ but $L^{\infty }(w_3) \not\hookrightarrow L^{1}$ (just take $f(x) = 1/x$), 
so $C$ is not bounded on $L^\infty(w_3)$ and $\supp CX = I$.

(b) It is easy to see that the space $Ces_{1}[0,1]$ is just a weighted $L^{1}(w)[0,1]$ space,
where $w(t)=\ln (1/t)$ for $0<t\leq 1$. Indeed, we have
\begin{equation}
\int_{0}^{1}(\frac{1}{x}\int_{0}^{x}\left\vert f(t)\right\vert \mathrm{d}t) 
\mathrm{d}x=\int_{0}^{1}(\int_{t}^{1}\frac{\mathrm{d}x}{x})\left\vert
f(t)\right\vert \mathrm{d}t=\int_{0}^{1}\left\vert f(t)\right\vert \ln (
\frac{1}{t})\mathrm{d}t,  \label{Ces1 = L1ln(1/t)}
\end{equation}
see \cite[Theorem 1 (a)]{AM09}. Therefore, despite the fact that the Ces\`aro operator 
$C$ is not bounded on $L^{1}[0,1]$ (cf. Theorem C), we see again that
$\supp(Ces_1[0,1]) = [0,1]$. Thus, if $f\in Ces_{1}[0,1]$ and $\supp(f)\subset [a,b]$, 
where $0<a<b<1$, then 
\begin{equation}
\ln (\frac{1}{b})\left\Vert f\right\Vert _{L^{1}[0,1]}\leq \left\Vert
f\right\Vert _{Ces_{1}[0,1]}\leq \ln (\frac{1}{a})\left\Vert f\right\Vert
_{L^{1}[0,1]},  \label{nierownosc1dlaL1[0,1]}
\end{equation}
see also \cite[Lemma 1, inequality (4)]{AM08}. Equality (\ref{Ces1 = L1ln(1/t)}) shows by the way 
that $Ces_{1}[0,\infty )=\{0\}$, cf. \cite[Theorem 1 (a)]{AM09}.
	
(c) Clearly, $Cop_1 \equiv L^1$ and $Cop_\infty \equiv L^1(1/t)$ because
$$ 
\norm{f}_{Cop_1} = \int_I (\int_x^{m(I)} \frac{\abs{f(t)}}{t} \mathrm{d}t) \mathrm{d}x 
= \int_I (\int_0^t \mathrm{d}x) \frac{\abs{f(t)}}{t} \mathrm{d}t = \norm{f}_{L^1},
$$
and
$$ 
\norm{f}_{Cop_\infty} = \sup\limits_{x \in I} \int_x^{m(I)} \frac{\abs{f(t)}}{t} \mathrm{d}t 
= \int_0^{m(I)} \frac{\abs{f(t)}}{t} \mathrm{d}t = \norm{f}_{L^1(1/t)}.
$$
Again, $\supp(Cop_\infty) = I$ but the Copson operator $C^*$ is not bounded on $L^\infty$, 
see Theorem C.	
	
(d) Let $X$ be a symmetric space on $[0,1]$ with $p(X) = 1$. Then the Ces\`aro operator $C$ 
is not bounded on $X$ but $CX \neq \{ 0 \}$, cf. Theorem C and D. For example, if $X$ is the 
Zygmund space $L\log L[0,1]$ (see \cite[Definition 6.1, p. 243)]{BS88}), then $p(X) = 1$ (see 
\cite[Theorem 6.5, p. 247)]{BS88}).
	
(e) Consider a symmetric space $X$ on $I$ such that the Copson operator $C^{\ast}$ is not 
bounded on $X$. Then, by Lemma \ref{Copson nontrivial}, $C^{\ast}X \neq \left\{ 0\right\}$, 
because $\chi _{\lbrack 0,\lambda ]}\in X$ for each $0 < \lambda < m(I)$. In particular, we can 
take $X = L^{\Phi}$, where $L^\Phi$ is the Orlicz space generated by the Orlicz function $\Phi$ 
which does not satisfy the suitable $\Delta _{2}$-condition. Since $\Phi \notin \Delta _{2},$ 
so $q(L^{\Phi}) = \infty $ (see \cite[Proposition 2.5, p. 139]{LT79} and \cite[Theorem 3.2, p. 22]{Ma85}) 
and consequently $C^{\ast}$ is not bounded on $X$, cf. Theorem C.
	
(f) Suppose that $X$ is the Orlicz space $L^{\Phi}$ generated by the Orlicz function 
$$
\Phi(x) = x \log (1 + x).
$$
First, note that $p(L^{\Phi}) = \alpha_{\Phi}$, where $\alpha_{\Phi}$ is the lower Orlicz--Matuszewska 
index of $\Phi$ (see \cite[Proposition 2.5, p. 139 and Remark 2, p. 140]{LT79}). Moreover, it is not 
difficult to calculate that $\alpha_{\Phi}=1$ (cf. \cite[pp. 7--21]{Ma85}). Consequently, the operator 
$C$ is not bounded on $X$, cf. Theorem C. We claim that $CX\neq \left\{ 0\right\}$. It is clear when 
$I=\left[ 0,1\right]$. In fact, $L^{\Phi}$ is a symmetric space, so $\chi _{\lbrack 0,\lambda ]}\in X$ for 
each $0 < \lambda < 1$ and we can apply Theorem D. If $I=[0,\infty)$, according to Theorem D, 
we need to show that $(f_{\lambda } \colon x\mapsto \frac{1}{x}\chi _{\lbrack \lambda ,\infty )}(x)) 
\in L^{\Phi}$ for some $\lambda > 0$. Recall that $f \in L^{\Phi}$, whenever 
$\int_{0}^\infty \Phi \left( \gamma \left\vert f\left( x\right) \right\vert \right) \mathrm{d}x < \infty $ 
for some $\gamma > 0$ (see \cite{Ma89}). We have
\begin{equation*}
\int_{0}^\infty \Phi \left( \left\vert f_{\lambda }\left( x\right) \right\vert
\right) \mathrm{d}x = \int_{\lambda }^{\infty }\frac{1}{x}\log \left( 1+\frac{1}{x} 
\right) \mathrm{d}x \leq \int_{\lambda}^{\infty}\frac{1}{x^2}\mathrm{d}x = \frac{1}{\lambda} < \infty,
\end{equation*}
and the claim follows.
\end{example}

\section{Copies of $\ell^1$, $\ell^\infty$, $L^1[0,1]$ and $L^\infty[0,1]$ in Ces\`aro, Copson 
and Tandori function spaces} \label{sekcja kopii l1}

The norms of the Ces\`aro and Copson function spaces are generated by a positive sublinear operator 
$T$, where $T$ stands for the Ces\`aro or the Copson operator, and by the norm of a Banach function 
space $X$. Thus, in a sense, the space $TX$ is a nontrivial mix of the space $L^1$ and $X$ and 
some similarities to both these spaces can be found in $TX$. We will make these statements more 
precise showing first that the Ces\`aro and Copson function spaces contain ``good" copies 
of $L^1[0,1]$ and $\ell^1$.

\begin{lemma} \label{complementedL1}
{\it Let $X$ be a Banach function space on $I$ such that $CX \neq \{ 0 \}$. 
Then there are numbers $0 < a < b < m(I)$ such that
\begin{equation}
\left\Vert \frac{1}{x}\chi _{\lbrack b,m(I))}(x)\right\Vert
_{X}\left\Vert f\right\Vert _{L^{1}(I)}\leq \left\Vert f\right\Vert
_{CX}\leq \left\Vert \frac{1}{x}\chi _{\lbrack a,m(I))}(x)\right\Vert
_{X}\left\Vert f\right\Vert _{L^{1}(I)},  \label{nierownosc1}
\end{equation}
for all $f\in CX$ with $\supp(f)\subset \lbrack a,b]$.
In particular, the space $CX$ contains a complemented copy of $L^{1}[0,1]$.}
\end{lemma}

\begin{proof}
We will give only a sketch of the proof because, in fact, this lemma is just
a reformulation of Proposition 2.2 in \cite{ALM17} (cf. also \cite[Theorem 5.1 (b)]{AM14}).

Let $I=[0,1].$ First of all, $\chi _{\lbrack \lambda ,1]}\in X$ for some $
0<\lambda <1$ due to nontriviality of the space $CX$, see Theorem D. Take $a = \lambda$
and choose a number $ b \in (a,1) $.
Then $\frac{1}{x}\chi _{\lbrack
a ,1]}\left( x\right) \in X$ and, from the ideal property, also $\frac{1}{x}\chi _{\lbrack
b ,1]}\left( x\right) \in X$. Now, for such numbers $0 < a < b < 1$ and $f\in CX$ with $ 
\supp(f)\subset \lbrack a,b]$ it is obvious that 
\begin{equation*}
\frac{1}{x}\left\Vert f\right\Vert _{L^{1}[0,1]}\chi _{\lbrack b,1]}(x)\leq 
\frac{1}{x}\int_{0}^{x}\left\vert f(t)\right\vert \mathrm{d}t\leq \frac{1}{x} 
\left\Vert f\right\Vert _{L^{1}[0,1]}\chi _{\lbrack a,1]}(x),
\end{equation*} 
for any $0<x\in I$. Thus, 
\begin{equation}
\left\Vert \frac{1}{x}\chi _{\lbrack b,1]}(x)\right\Vert _{X}\left\Vert
f\right\Vert _{L^{1}[0,1]} \leq \left\Vert f\right\Vert _{CX}=\left\Vert 
\frac{1}{x}\int_{0}^{x}\left\vert f(t)\right\vert \mathrm{d}t\right\Vert _{X}
 \leq \left\Vert \frac{1}{x}\chi _{\lbrack a,1]}(x)\right\Vert
_{X}\left\Vert f\right\Vert _{L^{1}[0,1]},
\end{equation} 
and (i) follows.

At this point, it is clear that $\{f \in L^1[0,1] \colon \supp(f) \subset [a,b] \} \simeq L^{1}[0,1]$
and this copy of $L^{1}[0,1]$ is in fact complemented because the projection 
$P \colon f\mapsto f\chi _{\lbrack a,b]}$ is bounded.

In the case when $I=[0,\infty )$ the proof is completely analogous.
\end{proof}

\begin{lemma} \label{complemented L1 in C*X}
{\it Let $X$ be a Banach function space on $I$ such that $C^*X \neq \{ 0 \}$.
Then there are numbers $0 < a < b < m(I)$ such that
\begin{equation} \label{nierownosc Copson}
\norm{\chi_{[0,a]}}_X \norm{f}_{L^1(1/t)(I)} \leq \norm{f}_{C^*X} \leq \norm{\chi_{[0,b]}}_X \norm{f}_{L^1(1/t)(I)},
\end{equation}
for $f\in C^*X$ with $\supp(f) \subset [a,b]$. In particular, the space $C^*X$ contains a 
complemented copy of $L^{1}[0,1]$.}
\end{lemma}
\begin{proof}
We will give the proof only if $I = [0,1]$. The remaining case is analogous.
	
Suppose $I = [0,1]$. Thanks to the assumption that the Copson function space $C^*X$ is 
nontrivial we get $\chi_{[0,\lambda]} \in X$ for some $0 < \lambda < 1$, see 
Lemma \ref{Copson nontrivial}. Take $b = \lambda$ and choose a number $a \in (0,b)$.
If $f\in C^*X$ and $\supp(f) \subset [a,b]$, then we have
\begin{equation*}
C^* \abs{f}(x) = \int^1_x \frac{\abs{f(t)}}{t} \mathrm{d}t \geq \int_a^b \frac{\abs{f(t)}}{t} \mathrm{d}t \chi_{[0,a]}(x)
= \norm{f}_{L^1(1/t)[0,1]} \chi_{[0,a]}(x).
\end{equation*}
Moreover,
\begin{equation*}
C^* \abs{f}(x) = \int^1_x \frac{\abs{f(t)}}{t} \mathrm{d}t \leq \int_a^b \frac{\abs{f(t)}}{t} \mathrm{d}t \chi_{[0,b]}(x)
= \norm{f}_{L^1(1/t)[0,1]} \chi_{[0,b]}(x).
\end{equation*}
Putting together the above inequalities we obtain (\ref{nierownosc Copson}).
	
The last part of this lemma is clear since
$$
\{ f \in L^1(1/t)[0,1] \colon \supp(f) \subset [a,b] \} = \{ f \in L^1[0,1] \colon \supp(f) \subset [a,b] \} \simeq L^1[0,1],
$$
whenever $0 < a < b < 1$ (because $\frac{1}{b} \norm{f\chi_{[a,b]}}_{L^1[0,1]} \leq 
\norm{f\chi_{[a,b]}}_{L^1(1/t)[0,1]} \leq \frac{1}{a} \norm{f\chi_{[a,b]}}_{L^1[0,1]}$ for $f \in L^1(1/t)[0,1]$) 
and it is enough to take the projection $P \colon f \mapsto f\chi_{[a,b]}$.
\end{proof}

If additionally the Ces\`aro or the Copson operator is bounded on the Banach function space $X$, 
then we can deduce a little stronger versions of Lemma \ref{complementedL1} and 
Lemma \ref{complemented L1 in C*X}, respectively. More precisely, if the Ces\`aro operator $C$ is 
bounded on $X$, then it follows from Theorem D that $\supp(CX) = I$ and consequently, 
Lemma \ref{complementedL1} holds true for all $0 < a < b < m(I)$. Of course, due to Corollary 
\ref{Copson bounded and embeddings}, analogous remark holds true also for every nontrivial 
Copson function space $C^*X$.
	
It is clear, that every nontrivial Ces\`aro and Copson function space contains also a complemented 
copy of $\ell^{1}$ (simply because the space $L^1[0,1]$ contains such a copy). Moreover, James's 
distortion theorem for $\ell^1$ states that a Banach space $X$ contains an isomorphic copy of 
$\ell^1$ if and only if it contains an almost isometric copy of $\ell^1$, that is, for every 
$0 < \varepsilon < 1$, there exists a sequence $(x_n) \subset X$ such that
$(1 - \varepsilon)\sum_{n=1}^\infty \abs{\alpha_n} \leq \norm{\sum_{n=1}^\infty \alpha_{n}x_n}_X 
\leq \sum_{n=1}^\infty \abs{\alpha_n}$, for all $\alpha = (\alpha_n) \in \ell^1$. Therefore, as 
an immediate conclusion from the complemented version of James's distortion theorem 
\cite[Theorem 2]{DRT98} we obtain the following result.

\begin{corollary} \label{CX-contains compl of l1} 
{\it Let $T = C$ or $T = C^*$. If $X$ is a Banach function space on $I$ such that $TX \neq \{ 0 \}$, 
then the space $TX$ contains a complemented almost isometric copy of $\ell ^{1}$. In particular, 
the space $TX$ is not reflexive.}
\end{corollary}

It turns out that we can prove even a stronger versions of our Lemmas \ref{complementedL1} 
and \ref{complemented L1 in C*X} and Proposition 2.2 from \cite{ALM17}.

\begin{theorem} \label{general renorming}
{\it Let $X$ be a Banach space and assume that $X$ contains a complemented copy of a Banach 
space $Z$. Then there exists an equivalent norm on $X$ such that $X$ contains an isometric 
copy of $Z$. In particular, if $X$ is a Banach function space and $T = C$ or $T = C^*$, then:
\begin{enumerate}
\item [(i)] the space $TX$ can be renormed to contains an isometric copy of $L^1[0,1]$, 
whenever $TX \neq \{ 0 \}$.
\item [(ii)] every nontrivial Tandori function space $\tandori{X}$ can be renormed to contain 
an isometric copy of $L^\infty[0,1]$.
\end{enumerate}
	}
\end{theorem}
\begin{proof}
First, let $P \colon X \rightarrow X$ be a projection onto $Y_1 \subset X$ and $T$ be an isomorphism 
from $Z$ onto $Y_1$. Consequently, we have the following diagrams
\begin{equation*}
	\begin{tikzcd}
	X \arrow[r, "P"] \arrow[dr, dashrightarrow]
	& Y_1 \arrow[d, "T^{-1}"]\\
	& Z
	\end{tikzcd}
	\quad \text{ and } \quad
	\begin{tikzcd}
	X \arrow[r, "\text{id}_X - P"] \arrow[dr, dashrightarrow]
	& Y_2 \arrow[d, equal]\\
	& Y_2
	\end{tikzcd}
\end{equation*}
that is, $X \simeq Y_1 \oplus Y_2 \simeq Z \oplus Y_2$. 
We will introduce a new norm $\vertiii{\cdot}_X$ on the space $X$ which is defined as
\begin{equation}
\vertiii{x}_X := \norm{T^{-1}Px}_Z + \norm{(\text{id}_X - P)x}_X.
\end{equation}
It turns out that this norm is equivalent to the original one. In fact, 
\begin{align*}
\vertiii{x}_X & \leq \norm{T^{-1}}_{Y \rightarrow Z}\norm{Px}_Y + (1 + \norm{P})\norm{x}_X \\
 & \leq \norm{T^{-1}}_{Y \rightarrow Z}\norm{P}\norm{x}_X + (1 + \norm{P})\norm{x}_X \\
& \leq (1 + \norm{P} + \norm{T^{-1}}_{Y \rightarrow Z}\norm{P})\norm{x}_X.
\end{align*}
On the other hand,
\begin{align*}
\norm{x}_X & = \norm{Px + (\text{id}_X - P)x}_X \\
& \leq \norm{Px}_Y + \norm{(\text{id}_X - P)x}_X \\
& \lesssim \norm{T^{-1}Px}_Z + \norm{(\text{id}_X - P)x}_X = \vertiii{x}_X.
\end{align*}
Combining the above inequalities, we see that the norm $\vertiii{\cdot}_X$ is equivalent 
to the norm $\norm{\cdot}_X$. Moreover, if $y \in Y$ then
\begin{equation}
\vertiii{y}_X = \norm{T^{-1}Py}_Z + \norm{(\text{id}_X - P)y}_X = \norm{T^{-1}Py}_Z = \norm{T^{-1}y}_Z,
\end{equation}
which means that $T \colon Z \rightarrow (Y, \vertiii{\cdot}_X)$ is an isometry.
	
(i) This is clear due to Lemmas \ref{complementedL1} and \ref{complemented L1 in C*X} and the 
first part of the proof.
	
(ii) Let's start with a simple observation. If $X$ is a Banach function spaces on $I$, then nontriviality 
of the space $\tandori{X}$ is equivalent to the statement that $\chi_{[0,\lambda]} \in X$ for some 
$0 < \lambda < m(I)$. In fact, if $\tandori{X} \neq \{ 0 \}$, so we can find an element $f \in \tandori{X}$ 
such that $\abs{f(x)} > 0$ for $x \in A \subset I$ and $m(A) > 0$. 
Setting $B_n := \{ x \in A \colon \abs{f(x)} > 1/n \}$, where $n \in \mathbb{N}$, we see that there exists 
$n_0 \in \mathbb{N}$ with $m(B_{n_0}) > 0$. Therefore, we have
$$
\frac{1}{n_0} \chi_{[0,m(B_{n_0})]} \leq \tandori{f\chi_{B_{n_0}}} \leq \tandori{f} \in X,
$$
that is, $\chi_{[0,\lambda]} \in X$ for $\lambda = m(B_{n_0})$. The second implication is clear. Now,
since $\tandori{X} \neq \{ 0 \}$, it follows that $(0,\lambda) \subset \supp(\tandori{X})$ for some 
$0 < \lambda < m(I)$. Using the same argument as in \cite[Proposition 2.2]{ALM17} but for 
$0 < a < b < \lambda$ we can prove that the Tandori space $\tandori{X}$ contains a complemented 
copy of $L^\infty[0,1]$. To finish the proof it is enough to apply once again the first part.
\end{proof}

Let us note that if $X$ is a Banach function space on $I$ with $CX \neq \{ 0 \}$,
then exactly as we just did in Theorem \ref{general renorming}, we get the following
diagram
\begin{equation}
	\begin{tikzcd}
		CX \arrow[r, "P"]
		\arrow[d, dashrightarrow]
		& Y_1 \arrow[d, "\text{id}"] \\
		L_1[0,1] \arrow[r, leftarrow, "Q"]
		& L_1[a,b]
	\end{tikzcd}
\end{equation}
where $P \colon f \mapsto f|_{[a,b]}$ for $0 < a < b < m(I)$ is a bounded projection,
$Y_1 = L^1[a,b] := \{f \in L^1(I) \colon \supp(f) \subset [a,b] \}$, the mapping $Q$ is a linear isometry
between $L^1[a,b]$ and $L^1[0,1]$ and
$$\vertiii{f}_{CX} = \norm{f\chi_{[a,b]}}_{L^1} + \norm{f\chi_{(0,a) \cup (b,m(I))}}_{CX},$$
for $f \in CX$. Now, if we take $X = L^p$ for $1 \leq p < \infty$ if $I = [0,1]$ and $1 < p < \infty$ 
if $I = [0,\infty)$, then we obtain the Astashkin--Maligranda result from \cite[Lemma 4]{AM13b} 
concerning analogous renormings of classical Ces\`aro function spaces $Ces_p$.
\smallskip

Recall that a Banach function space $X$ contains an {\it order asymptotically isometric copy of $\ell ^{1}$},
whenever there is a sequence $\left(f_{n}\right) \subset X$ with pairwise disjoint supports and a sequence
$\left( \varepsilon _{n}\right) \subset \left( 0,1\right) $ such that $\varepsilon _{n}\rightarrow 0$ and
\begin{equation}
\sum_{n=1}^{\infty }\left( 1-\varepsilon _{n}\right) \left\vert \alpha
_{n}\right\vert \leq \left\Vert \sum_{n=1}^{\infty }\alpha
_{n}f_{n}\right\Vert _{X}\leq \sum_{n=1}^{\infty }\left\vert \alpha
_{n}\right\vert,
\end{equation}
for each $\alpha =\left( \alpha _{n}\right) \in \ell^{1}$. This notion was introduced by Dowling--Lennard 
in \cite[Definition 1.1]{DL97} and used to show that every nonreflexive subspace of $L^1[0,1]$ 
fails the fixed point property.

The notion of an asymptotically isometric copy of $\ell^1$ is closely related to an almost isometric copy of $\ell^1$
and consequently to James's distortion theorem (see \cite[Question]{DJLT97} and \cite[p. 270]{KS01}). However,
Dowling--Johnson--Lennard--Turett \cite{DJLT97} gave an example of a renorming of the space $\ell^1$ which 
contains no asymptotically isometric copy of $\ell^1$. We will further extend the class of spaces which contains 
an asymptotically isometric copy of $\ell^1$ showing that nontrivial Ces\`aro and Copson function spaces 
always contain such a copy. Note that for $Ces_p$-spaces such a claim has been proved in 
\cite[Theorems 1 and 2]{AM08}. It would seem naturally to look for a generalization of this result for symmetric 
spaces first (or even first for Orlicz spaces). Amazingly, it turns out that symmetry of the space $X$ is not important in our proof.

Before giving the proof, for $X$ being a Banach function space on $I$, let us define a function $F_{X}$ 
as follows
\begin{equation}
F_{X}:=F_{X[0,1]} \colon I\ni \lambda \mapsto \left\Vert \frac{1}{x}\chi
_{(\lambda ,1]}(x)\right\Vert _{X}\in \lbrack 0,\infty ],  \label{F}
\end{equation} 
if $I=[0,1]$, and
\begin{equation}
F_{X}:=F_{X[0,\infty )} \colon I\ni \lambda \mapsto \left\Vert \frac{1}{x}\chi
_{(\lambda ,\infty )}(x)\right\Vert _{X}\in \lbrack 0,\infty ],
\label{G}
\end{equation} 
if $I=[0,\infty )$.

\begin{theorem} \label{AIC-l1} \label{as kopia l1}
{\it Let $X$ be a Banach function space on $I$ such that the Ces\`aro function space $CX$ is nontrivial.
	Then the space $CX$ contains an order asymptotically isometric copy of $\ell ^{1}$.}
\end{theorem}
\begin{proof}
$1^0$. Suppose $I=[0,1]$. Since $CX \neq \{ 0 \}$, so $\chi_{[\lambda_0, 1]} \in X$ for some 
$0 < \lambda_0 < 1$, see Theorem D. For each $\lambda_0 < a < 1$ set
\begin{equation*}
\Omega _{a} := \{ \lambda \in (\lambda_0, 1) : F_{X}(\lambda )=F_{X}(a)\}.
\end{equation*} 
Of course, $\text{card} \left( \Omega _{a}\right) \geq 1$. Let us now consider the
following two cases.

(a) Assume that $\text{card}\left( \Omega_{a}\right) =1$ for every $a \in (\lambda_0, 1)$.
Obviously, the function $F_{X}$ is nonincreasing in the interval $[\lambda_0, 1]$, whence it
contains at most countably many points of discontinuity. Let $\lambda_0 < a_{0} < 1$ be
a point of continuity of the function $F_{X}$. Take a sequence $(a_{n})\subset
(\lambda_0,a_{0})$ such that $a_{n}\uparrow a_{0}$ as $n\rightarrow
\infty $ and put 
\begin{equation*}
g_{n}:=\frac{\chi _{(a_{n},a_{n+1})}}{\left\Vert \chi
_{(a_{n},a_{n+1})}\right\Vert _{CX}}.
\end{equation*}
From the definition $\supp(g_{n})=(a_{n},a_{n+1})\subset (a_{n},a_{0})$ and 
$ \supp(g_{n})\cap \supp(g_{m})=\emptyset $ if $n\neq m$, $m,n\in \mathbb{N}$.
Using the right-hand side of the estimate (\ref{nierownosc1}) we have 
\begin{align}
\left\Vert \chi _{(a_{n},a_{n+1})}\right\Vert _{CX}& \leq \left\Vert \frac{1 
}{x}\chi _{(a_{n},1]}(x)\right\Vert _{X}\left\Vert \chi
_{(a_{n},a_{n+1})}\right\Vert _{L^{1}[0,1]}  \label{nierownosc2b} \\
& =\left\Vert \frac{1}{x}\chi _{(a_{n},1]}(x)\right\Vert
_{X}(a_{n+1}-a_{n})=F_{X}(a_{n})(a_{n+1}-a_{n}).
\end{align} 
Furthermore, using left-hand side of (\ref{nierownosc1}) and the above
estimate, since elements $g_{n}$ are mutually disjoint, we obtain 
\begin{align*}
\left\Vert \sum_{n=1}^{\infty }\alpha _{n}g_{n}\right\Vert _{CX}& \geq
\left\Vert \frac{1}{x}\chi _{\lbrack a_{0},1]}(x)\right\Vert _{X}\left\Vert
\sum_{n=1}^{\infty }\alpha _{n}g_{n}\right\Vert _{L^{1}[0,1]} \\
& =F_{X}(a_{0})\left\Vert \sum_{n=1}^{\infty }\alpha _{n}g_{n}\right\Vert
_{L^{1}[0,1]}=F_{X}(a_{0})\sum_{n=1}^{\infty }\frac{\left\vert \alpha
_{n}\right\vert \left\Vert \chi _{(a_{n},a_{n+1})}\right\Vert _{L^{1}[0,1]}}{ 
\left\Vert \chi _{(a_{n},a_{n+1})}\right\Vert _{CX}} \\
& \geq F_{X}(a_{0})\sum_{n=1}^{\infty }\frac{\left\vert \alpha
_{n}\right\vert (a_{n+1}-a_{n})}{F_{X}(a_{n})(a_{n+1}-a_{n})} 
=\sum_{n=1}^{\infty }\frac{F_{X}(a_{0})}{F_{X}(a_{n})}\left\vert \alpha
_{n}\right\vert,
\end{align*} 
for each $\alpha =\left( \alpha _{n}\right) \in \ell^{1}.$ Denote 
\begin{equation*}
\theta _{n}:=\frac{F_{X}(a_{0})}{F_{X}(a_{n})}.
\end{equation*} 
Since $\text{card}\left( \Omega _{a_{0}}\right) = 1$, it follows that 
\begin{equation*}
F_{X}(a_{n})=\left\Vert \frac{1}{x}\chi _{(a_{n},1]}(x)\right\Vert
_{X}>\left\Vert \frac{1}{x}\chi _{(a_{0},1]}(x)\right\Vert _{X}=F_{X}(a_{0}).
\end{equation*} 
Consequently, $(\theta _{n})\subset (0,1)$ and, thanks to continuity of the
function $F_{X}$ at the point $a_{0}$, we have that $\theta _{n}\rightarrow
1 $ as $n\rightarrow \infty $. Finally, put 
\begin{equation*}
\varepsilon _{n}:=1-\theta _{n}.
\end{equation*} 
Then $(\varepsilon _{n})\subset (0,1)$, $\varepsilon _{n}\rightarrow 0$ as $%
n\rightarrow \infty $ and 
\begin{equation}
\left\Vert \sum_{n=1}^{\infty }\alpha _{n}g_{n}\right\Vert _{CX}\geq
\sum_{n=1}^{\infty }(1-\epsilon _{n})\left\vert \alpha _{n}\right\vert .
\label{kopial1nierownosc1}
\end{equation}
The second of the estimates we need is obvious. Note that $\left\Vert
g_{n}\right\Vert _{CX}=1$, so 
\begin{equation}
\left\Vert \sum_{n=1}^{\infty }\alpha _{n}g_{n}\right\Vert _{CX}\leq
\sum_{n=1}^{\infty }\left\vert \alpha _{n}\right\vert \left\Vert
g_{n}\right\Vert _{CX}=\sum_{n=1}^{\infty }\left\vert \alpha _{n}\right\vert. \label{kopial1nierownosc2}
\end{equation}
Thus, combining the inequalities (\ref{kopial1nierownosc1}) and (\ref 
{kopial1nierownosc2}), we finish the proof in that case.

(b) Assume that there is $a \in (\lambda_0, 1)$ with $\text{card}\left( \Omega
_{a}\right) >1.$ Therefore, there are numbers $a_1, a_2 \in (\lambda_0, 1)$
such that $a_{1}\neq a_{2},$ say $a_{1}<a_{2},$ and $F_{X}\left(
a_{1}\right) =F_{X}\left( a_{2}\right) .$ Thus, for each number $a_3$ with $a_{1}< a_3 < a_{2},$ by
the monotonicity of the norm, we have 
\begin{equation*}
F_{X}(a_{1})\geq F_{X}\left( a_3\right) \geq F_{X}\left( a_{2}\right),
\end{equation*}%
which means the function $F_{X}$ is constant on the interval $\left(
a_{1},a_{2}\right) ,$ i.e. $\left( a_{1},a_{2}\right) \subset \Omega _{a}.$
Following the same way as in case $\left( a\right) $ we get easily that the
space $CX$ contains even an order isometric copy of $\ell ^{1}$.

$2^0$. The proof when $I=[0,\infty )$ is the same as in the previous case.
The only difference, of course, lies in the consideration of the function
\begin{equation}
F_{X} \colon [0,\infty )\ni \lambda \mapsto \left\Vert \frac{1}{x}\chi _{(\lambda
,\infty )}(x)\right\Vert _{X}\in (0,\infty].
\end{equation}
\end{proof}

It is clear that due to the similarities occurring in Lemma \ref{complementedL1} and 
Lemma \ref{complemented L1 in C*X}, a result analogous to Theorem \ref{AIC-l1} will be 
rather expected also in the case of the Copson function spaces. It will be convenient to 
start with the following natural modification of the previously introduced function $F_X$, 
namely,
\begin{equation}
G_X := G_{X(I)} \colon I \ni \lambda \mapsto \norm{\chi_{[0,\lambda]}}_{X(I)} \in (0,\infty ],
\end{equation}
where $X$ is a Banach function space on $I$.

\begin{theorem} \label{AIC-l1 in C^*X}
{\it Let $X$ be a Banach function space on $I$ such that the Copson function space $C^*X$ is nontrivial.
	Then the space $C^*X$ contains an order asymptotically isometric copy of $\ell ^{1}$.}
\end{theorem}
\begin{proof}
With minor changes the proof is similar as that of Theorem \ref{AIC-l1}.
Note, however, that the structure of the proof itself seems to be dual to the previous one.
Details are provided for the convenience of the reader.
	
$1^0$. Suppose $I=[0,1]$. Since $C^*X \neq \{ 0 \}$, there is $0 < \lambda_0 < 1$ with 
$\chi_{[0, \lambda_0]} \in X$, see Lemma \ref{Copson nontrivial}. For each 
$b \in (0,\lambda_0)$ set
\begin{equation*}
\Omega _{b} := \{ \lambda \in (0, \lambda_0) : G_{X}(\lambda ) = G_{X}(b)\}.
\end{equation*}
Of course, $\text{card} \left( \Omega _{b}\right) \geq 1$. Let us now consider the 
following two cases.
	
(a) Assume that $\text{card}\left( \Omega _{b}\right) = 1$ for every $0 < b < 1$. Obviously, the 
function $G_{X}$ is nondecreasing on the interval $[0, \lambda_0]$, whence it contains at 
most countably many points of discontinuity. Let $b_{0} \in (0,\lambda_0)$ be a point of 
continuity of the function $G_{X}$. Take a sequence $(b_{n})\subset (0,\lambda_{0})$ such 
that $b_{n}\downarrow b_{0}$ as $n\rightarrow \infty $ and put 
\begin{equation*}
h_{n}:=\frac{\chi _{(b_{n+1},b_{n})}}{\left\Vert \chi_{(b_{n+1},b_{n})}\right\Vert _{C^*X}}.
\end{equation*}
From the definition $\supp(h_{n})=(b_{n+1},b_{n})\subset (b_{0},b_{n})$ and 
$\supp(h_{n})\cap \supp(h_{m})=\emptyset $ if $n\neq m$, $m,n\in \mathbb{N}$.
Using the right-hand side of the estimate (\ref{nierownosc Copson}) we have 
\begin{align}
\left\Vert \chi _{(b_{n+1},b_{n})}\right\Vert _{C^*X} & \leq 
\left\Vert \chi _{[0,b_n]}\right\Vert _{X}\left\Vert \chi_{(b_{n+1},b_{n})}\right\Vert _{L^{1}(1/t)[0,1]} 
= G_X(b_n) \norm{\chi_{(b_{n+1},b_{n})}}_{L^{1}(1/t)[0,1]}.
\end{align}
Furthermore, using left-hand side of (\ref{nierownosc Copson}) and the above estimate, since 
elements $h_{n}$ are mutually disjoint, we obtain 
\begin{align*}
\left\Vert \sum_{n=1}^{\infty }\alpha _{n}h_{n}\right\Vert _{C^*X}& \geq
\left\Vert \chi _{[0, b_0]}\right\Vert _{X}\left\Vert
\sum_{n=1}^{\infty }\alpha _{n}h_{n}\right\Vert _{L^{1}(1/t)[0,1]} \\
& = G_{X}(b_{0})\left\Vert \sum_{n=1}^{\infty }\alpha _{n}h_{n}\right\Vert
_{L^{1}(1/t)[0,1]} = G_{X}(b_{0}) \sum_{n=1}^{\infty }\frac{\left\vert \alpha
_{n}\right\vert \left\Vert \chi _{(b_{n+1},b_{n})}\right\Vert _{L^{1}(1/t)[0,1]}}{
\left\Vert \chi _{(b_{n+1},b_{n})}\right\Vert _{C^*X}} \\
& \geq G_{X}(b_{0})\sum_{n=1}^{\infty }\frac{\abs{\alpha_n}\norm{\chi_{(b_{n+1},b_{n})}}_{L^{1}(1/t)[0,1]}}{G_{X}(b_{n})\norm{\chi_{(b_{n+1},b_{n})}}_{L^{1}(1/t)[0,1]}}
=\sum_{n=1}^{\infty }\frac{G_{X}(b_{0})}{G_{X}(b_{n})}\left\vert \alpha_{n}\right\vert,
\end{align*}
for each $\alpha =\left( \alpha _{n}\right) \in \ell^{1}.$ Denote 
\begin{equation*}
\theta _{n}:=\frac{G_{X}(b_{0})}{G_{X}(b_{n})}.
\end{equation*}
Since $\text{card}\left( \Omega _{b_{0}}\right) = 1$, it follows that 
\begin{equation*}
G_{X}(b_{n}) = \norm{\chi_{[0,b_n]}}_X > \norm{\chi_{[0,b_0]}}_X = G_{X}(b_{0}).
\end{equation*}
Consequently, $(\theta _{n})\subset (0,1)$ and, thanks to continuity of the function $G_{X}$ 
at the point $b_{0}$, we have that $\theta _{n}\rightarrow 1 $ as $n\rightarrow \infty $. Finally, 
put 
\begin{equation*}
\varepsilon _{n} := 1-\theta _{n}.
\end{equation*}
Then $(\varepsilon _{n})\subset (0,1)$, $\varepsilon _{n}\rightarrow 0$ as $n\rightarrow \infty $ and 
\begin{equation}
\left\Vert \sum_{n=1}^{\infty }\alpha _{n}h_{n}\right\Vert _{C^*X}\geq
\sum_{n=1}^{\infty }(1-\epsilon _{n})\left\vert \alpha _{n}\right\vert .
\label{kopial1nierownosc3}
\end{equation}
The second of the estimates we need is trivial. Note that $\left\Vert h_{n}\right\Vert _{C^*X} = 1$, 
so 
\begin{equation}
\left\Vert \sum_{n=1}^{\infty }\alpha _{n}h_{n}\right\Vert _{C^*X}\leq
\sum_{n=1}^{\infty }\left\vert \alpha _{n}\right\vert \left\Vert
h_{n}\right\Vert _{C^*X}=\sum_{n=1}^{\infty }\left\vert \alpha _{n}\right\vert.  
\label{kopial1nierownosc4}
\end{equation}
Thus, combining the inequalities (\ref{kopial1nierownosc3}) and (\ref {kopial1nierownosc4}), 
we finish the proof in that case.
	
(b) Assume that there is $b \in (0, \lambda_0)$ with $\text{card}\left( \Omega_{b}\right) > 1$. Therefore, 
there are numbers $b_1, b_2 \in (0,1)$ such that $b_{1}\neq b_{2},$ say $b_{1}<b_{2},$ and 
$G_{X}\left(b_{1}\right) = G_{X}\left(b_{2}\right)$. Thus, for each number $b_3$ with 
$b_{1}< b_3 < b_{2},$ by the monotonicity of the norm, we have 
\begin{equation*}
G_{X}(b_{1}) \leq G_{X}\left( b_3\right) \leq G_{X}\left(b_{2}\right),
\end{equation*}
which means the function $G_{X}$ is constant on the interval $\left(b_{1},b_{2}\right)$, 
i.e., $\left(b_{1},b_{2}\right) \subset \Omega _{b}$. Following the same way as in case (a) 
we get easily that the space $C^*X$ contains even an order isometric copy of $\ell ^{1}$.
	
$2^0$. If $I = [0,\infty)$, we follow the same way as in $1^0$.
\end{proof}

In the context of Theorem \ref{general renorming} and Theorem \ref{AIC-l1} a natural question 
arises: maybe the Ces\`aro function space $(CX, \norm{\cdot}_{CX})$ always contains an isometric 
copy of $\ell^1$ or $L^1[0,1]$? In general, the answer is no. Indeed, if the Banach function space 
$X$ is rotund, then the space $CX$ is also rotund, see \cite[Lemma 2]{Ki-Kol-rot}. Therefore, for example, 
if we take $X = L^p$ for $1 < p < \infty$, then $Ces_p$ is rotund. Consequently, the space $Ces_p$ 
cannot contain an isometric copy of neither $\ell^1$ nor $L^{1}[0,1]$, because they are not rotund. 
However, let us remind that the space $(CX, \norm{\cdot}_{CX})$ can always be renormed to contain an isometric copy of 
$L^{1}[0,1]$, cf. Theorem \ref{general renorming} (i).

Theorem \ref{as kopia l1} gives also some information about the generalized Tandori function spaces 
$\widetilde{\mathit{X}}$. In short, they are quite similar to $L^\infty$.

\begin{corollary}
{\it Let $X$ be Banach function space on $I=[0,\infty )$ such that the space $X^{\prime}$ is order continuous 
and $X$ has the Fatou property (which is true, for example, if $X$ is a reflexive space). Assume also that the 
Copson operator $C^{\ast } \colon X\rightarrow X$ is bounded and the dilation operator 
$\sigma_{\tau} \colon X\rightarrow X$ is bounded for some $\tau > 1$. Then the Tandori 
function space $\widetilde{\mathit{X}}$ contains an isomorphic copy of $L^{1}[0,1]$ and 
$C[0,1]^{\ast }$.}
\end{corollary}

\begin{proof}
Because $X^{\prime }\in (OC)$ then $C\left( X^{\prime }\right) \in (OC).$ Note that, since 
$X\in \left( FP\right),$ $C^{\ast } \colon X\rightarrow X$ if and only if 
$C \colon X^{\prime }\rightarrow X^{\prime }$ and $\sigma_{\tau} \colon X\rightarrow X$ if and
only if $\sigma_{1/\tau} \colon X^{\prime }\rightarrow X^{\prime }$, see for example \cite[Remark 1]{KLM-2019}.
Consequently, applying Theorem A (\ref{Theorem A na [0,infty]})  for the space $X^{\prime },$ we get
\begin{equation}
\left( C\left( X^{\prime }\right) \right) ^{\ast }=(C(X^{\prime }))^{\prime
}=\widetilde{\mathit{X^{\prime \prime }}}=\widetilde{\mathit{X}},  \label{tan}
\end{equation}
with equivalent norms. The space $C(X^{\prime }) $ contains an order asymptotically isometric 
copy of $\ell ^{1}$ via Theorem \ref{as kopia l1}. By\ the Dilworth--Girardi--Hagler result 
\cite[Theorem 2]{DGH00}, the dual space $(C\left( X^{\prime }\right))^{\ast }$ contains an isometric 
copy of $L^{1}[0,1]$ and an isometric copy of $C[0,1]^{\ast }.$ Thus, by equality (\ref{tan}), the 
Tandori space $\widetilde{\mathit{X}}$ contains an isomorphic copy of $L^{1}[0,1]$ and $C[0,1]^{\ast }$.
\end{proof}

The similarity between the Tandori function spaces $\tandori{X}$ and $L^\infty$ becomes even 
clearer in the context of the following result.

\begin{proposition} \label{isometric l^infty in Tandori}
{\it Let $X$ be a Banach function space on $I$ such that the Tandori function space $\widetilde{\mathit{X}}$ is nontrivial.
Then the Tandori function space $\widetilde{\mathit{X}}$ contains an order isomorphically isometric copy of $\ell^{\infty }$.}
\end{proposition}
\begin{proof}
Since $\widetilde{\mathit{X}}\neq \left\{ 0\right\}$, it follows that there exists $0 < a \in I$ such that 
$\chi_{\lbrack 0,a)}\in X$ (see \cite[Theorem 1 (c)]{LM15a}). Put 
\begin{equation*}
f_0 = \frac{\chi _{\lbrack 0,a)}}{\left\Vert \chi _{\lbrack 0,a)}\right\Vert _{X}}.
\end{equation*}
Let $a_{n}=(1-\frac{1}{2^{n+1}})a$ and $A_{n}=\left( a_{n}-\delta _{n},a_{n}+\delta _{n}\right)$, where 
$\delta_{n} = (a_{n+1}-a_{n}) /2$ and $n \in \mathbb{N}$. Denote 
\begin{equation*}
B_{1}=[0,a/2)\cup \bigcup_{n=1}^{\infty }A_{n}.
\end{equation*}
Then the set $[0,a)\backslash B_{1}$ consists of infinitely many pairwise disjoint intervals, say,
$[0,a)\setminus B_{1}=\bigcup_{n=1}^{\infty}C_{n}^{\left( 1\right) }$. Let
\begin{equation*}
B_{2}=C_{1}^{\left( 1\right) }\cup C_{3}^{\left( 1\right) }\cup
C_{5}^{\left( 1\right) }\cup \ldots=\bigcup_{n=1}^{\infty }C_{2n-1}^{\left(1\right) }.
\end{equation*}
Again, the set $[0,a)\backslash \left( B_{1}\cup B_{2}\right) $ consists of infinitely many pairwise 
disjoint intervals, say, 
$[0,a)\setminus \left(B_{1}\cup B_{2}\right) =\bigcup_{n=1}^{\infty }C_{n}^{\left( 2\right) }$.
Next, let 
\begin{equation*}
B_{3}=C_{1}^{\left( 2\right) }\cup C_{3}^{\left( 2\right) }\cup C_{5}^{\left( 2\right) }\cup \ldots
=\bigcup_{n=1}^{\infty }C_{2n-1}^{\left(2\right) }.
\end{equation*}
We proceed analogously, defining the sequence of sets $\left( B_{n}\right)_{n=1}^{\infty }$. 
Put $f_{n} := f_0\chi _{B_{n}}$. Note that
\begin{equation*}
0\leq f_{n}\leq f_0 \text{ and } \supp(f_{n}) \cap \supp(f_{m}) = \emptyset \text{ for each } n \neq m.
\end{equation*}
Moreover, 
\begin{equation*}
\widetilde{f_{n}}=\widetilde{f_0} = f_0,
\end{equation*}
whence
\begin{equation*}
\left\Vert f_{n}\right\Vert _{\widetilde{X}}=\left\Vert \widetilde{f_{n}}
\right\Vert _{X}=\left\Vert \widetilde{f_0}\right\Vert _{X} = \left\Vert
f_0\right\Vert _{X}=1\text{ and }\left\Vert f_0\right\Vert _{\widetilde{X}
}=\left\Vert \widetilde{f_0}\right\Vert _{X}=\left\Vert f_0\right\Vert _{X}=1.
\end{equation*}
Applying Theorem 1 from \cite{Hu98} we conclude that the space $\widetilde{X}$ contains 
an order isomorphically isometric copy of $\ell ^{\infty}$.
\end{proof}

The problem of describing the Ces\`aro--Orlicz function spaces containing an order isomorphically 
isometric copy of $\ell^\infty$ has been considered in \cite{Ki-Kol-isom}. Although formally the 
case of the space $Ces_\infty$ (which is not isomorphic to $\ell^{\infty}$ -- see \cite[Theorem 7]{AM09}) 
has been excluded there, it follows that the argument used in \cite[Theorems 3 and 4 case 
(B2)]{Ki-Kol-isom} can be applied to get the below result. We will give, however, the direct and simple 
proof without referring to the structure of the Orlicz spaces.

\begin{proposition} \label{Ces_infty isometric copy of l8}
{\it The space $Ces_\infty$ contains an order isomorphically isometric copy of $\ell^\infty$.}
\end{proposition}
\begin{proof}
For a start, let us recall that
\begin{equation} \label{Ces_infty_a}
(Ces_\infty)_a = \{ f \in Ces_\infty : \lim\limits_{x\rightarrow 0^+, \infty} \frac{1}{x}\int_0^x \abs{f(t)} \mathrm{d}t = 0 \},
\end{equation}
see \cite[Remark 19]{KT17}. Put $f_0 := \chi_{[0,1]}$. Then $\norm{f_0}_{Ces_\infty} = 1$ and 
$\text{dist}(f_0, (Ces_\infty)_a) = 1$. In fact, it follows from (\ref{Ces_infty_a}) that
\begin{align*}
\text{dist}(f_0, (Ces_\infty)_a) & := \inf\limits_{h \in (Ces_\infty)_a} \norm{f_0 - h}_{Ces_\infty} \\
& = \inf\limits_{h \in (Ces_\infty)_a} \sup\limits_{0 \leq t \in I} C\abs{f_0 - h}(t) \\
& \geq \inf\limits_{h \in (Ces_\infty)_a} \sup\limits_{0 \leq t \in I} (C\abs{f_0}(t) - C\abs{h}(t)) \\
& \geq \inf\limits_{h \in (Ces_\infty)_a} \lim\limits_{t \rightarrow 0^+} ( C\abs{f_0}(t) - C\abs{h}(t) ) 
= \lim\limits_{t \rightarrow 0^+} C\abs{f_0}(t) = 1.
\end{align*}
Now, we can once again use Theorem 2 from \cite{Hu98} and finish the proof.
\end{proof}

\section{On differences in the Ces\`aro construction on $[0,1]$ and $[0,\infty)$}

We begin with a short discussion. Recall that following the standard definitions we should define a 
truncation of the Banach function space $X$ on $I$ to a set $A \subset I$ as 
$X|_{A} = \{f \in X \colon \supp(f) \subset A \}$, where $\supp(f) := \{x \in I \colon f(x) \neq 0 \}$. 
However, when applying the Ces\`aro construction to the truncated space the situation is more 
delicate. If we would like to follow the above definition, then we should write that 
$C(X[0,\infty)|_{[0,1]}))$ is the space of all functions $f \in L^0[0,\infty)$ (because the space 
$X[0,\infty)|_{[0,1]}$ contains functions from $L^0[0,\infty)$) such that $C\abs{f} \in X[0,\infty)|_{[0,1]}$, 
i.e., $f \in X[0,\infty)$ and $\supp(C\abs{f}) \subset [0,1]$.
But the last condition is never satified, whenever $0 \neq f \in L^0[0,\infty)$, so this is meaningless.
That is why we will proceed in the following way. For a Banach function space $X$ on $[0,\infty)$ 
and a subset $A \subset [0,\infty)$ we define the {\it truncation of the space $X[0,\infty)$ to a set} $A$ 
as
\begin{equation} \label{2}
X[0,\infty)|_{A} := \{f \in L^0[0,1] \colon f\chi_{A} \in X[0,\infty) \},
\end{equation}
with the norm $\norm{f}_{X[0,\infty)|_{A}} = \norm{f\chi_{A}}_{X[0,\infty)}$ (that is, we look at 
an element $f$ as a function defined on $[0,1]$, but compute its norm as if it was still defined on 
$[0,\infty)$ and equal zero on $[0,\infty) \setminus A$), if $A \subset [0,1]$ and
\begin{equation} \label{22}
X[0,\infty)|_{A} := \{f \in X[0,\infty) \colon \supp(f) \subset A \},
\end{equation}
with the norm $\norm{f}_{X[0,\infty)|_{A}} = \norm{f}_{X[0,\infty)}$, if $A \not\subset [0,1]$. 
Moreover, if $X$ is a Banach function space on $[0,1]$ and $A \subset [0,1]$, then we can 
use the definition given at the beginning, that is, the {\it truncation of the space $X[0,1]$ 
to the set} $A$ as
\begin{equation}
X[0,1]|_{A} := \{f \in X[0,1] \colon \supp(f) \subset A \},
\end{equation}
with the norm $\norm{f}_{X[0,1]|_{A}} = \norm{f}_{X[0,1]}$. Note that if we use the following notation
\begin{equation}
\mathfrak{X}[0,1] := X[0,\infty)|_{[0,1]} = \{ f \in L^0[0,1] \colon f\chi_{[0,1]} \in X[0,\infty) \},
\end{equation}
with the norm $\norm{f}_{\mathfrak{X}} = \norm{f\chi_{[0,1]}}_{X[0,\infty)}$, then the functor 
$X \mapsto \mathfrak{X}$ is a way to associate to the space $X$ defined on $I = [0,\infty)$ its ``natural" 
counterpart defined on $I = [0,1]$. Observe also that if $A \subset [0,1]$, then the space $X[0,\infty)|_{A}$ 
is actually isometrically isomorphic to the space $Y := \{f \in X[0,\infty) \colon \supp(f) \subset A \}$ with 
the norm $\norm{f}_Y = \norm{f}_{X[0,\infty)}$ via the mapping $J \colon f \mapsto f\chi_{A}$. Moreover, 
it is clear, that $X[0,\infty)|_{[0,1]}$ is a symmetric space, whenever $X[0,\infty)$ is a symmetric space.

Clearly, $L^{p}[0,\infty)|_{[0,1]} \equiv L^p[0,1]$. In contrast, for the Ces\`aro function spaces the situation 
is quite different. In fact, if $f\in Ces_{p}[0,\infty )$ and $\supp(f)\subset \lbrack 0,1]$, then 
\begin{equation} \label{CX[0,8) = CX[0,1] n L^1[0,1]}
\left\Vert f\right\Vert _{Ces_{p}[0,\infty )}^{p}=\left\Vert f\right\Vert
_{Ces_{p}[0,1]}^{p}+\frac{1}{p-1}\left\Vert f\right\Vert _{L^{1}[0,1]}^{p},
\end{equation}
i.e., $Ces_{p}[0,\infty )|_{\left[ 0,1\right] } = Ces_{p}\left[ 0,1\right] \cap L^{1}\left[ 0,1\right] $ for 
$1 < p < \infty $, see \cite[Remark 5]{AM09}. In the next lemma we will show an analogue of equality 
(\ref{CX[0,8) = CX[0,1] n L^1[0,1]}) for abstract Ces\`aro function spaces $CX$. However, let us 
notice now that according to definitions (\ref{2}) and (\ref{22}), we have
\begin{equation}
CX[0,\infty )|_{\left[ 0,1 \right] } = \left( CX[0,\infty ) \right) |_{\left[ 0,1 \right] } 
:= \left\{ f \in L^0[0,1] \colon C(\abs{f}\chi_{[0,1]}) \in X[0,\infty) \right\},
\end{equation}
with the norm $\norm{f}_{CX[0,\infty )|_{[0,1]}} = \norm{C(\abs{f}\chi_{[0,1]})}_{X[0,\infty )}$, and
\begin{equation} \label{Konstrukcja CX dla obcietego X}
C(X[0,\infty )|_{[0,1]})) := C(\mathfrak{X}[0,1]) = \left\{ f \in L^0[0,1] \colon \left( C\abs{f} \right) 
\chi_{(0,1]} \in X[0,\infty )\right\},
\end{equation}
with the norm $\norm{f}_{C(X[0,\infty )|_{[0,1]}))} = \norm{(C\abs{f})\chi_{(0,1]}}_{X[0,\infty)}$
(cf. with the definition of the Ces\`aro operator $C$ and the abstract Ces\`aro space $CX$).
Roughly speaking, if $X$ is a Banach function space on $[0,\infty)$, then we always have two 
ways (in general, nonequivalent) to obtain the Ces\`aro function space on $[0,1]$ -- first applying 
the functor $X \mapsto \mathfrak{X}$ and then the Ces\`aro construction or vice versa.
All this also means that the Ces\`aro construction $X \mapsto CX$ is significantly different for 
$I = [0,1]$ and $I = [0,\infty)$.

The equality (\ref{CXrownanie1}) from Lemma \ref{obciecie} below is an abstract version of equality 
(\ref{CX[0,8) = CX[0,1] n L^1[0,1]}) from \cite{AM09} and means that the functor 
$X \mapsto \mathfrak{X}$ does not comutate in general with the Ces\`aro construction
$X \mapsto CX$ (that is, $\mathfrak{CX}[0,1] = C\mathfrak{X}[0,1] \cap L^1[0,1]$).
This result explains also, in some sense, a rather suprising difference in the description 
of the K\"othe duality of Ces\`aro function spaces $CX$ on $I=[0,1]$ and on $I=[0,\infty)$, 
cf. \cite[Theorems 3, 5 and 6]{LM15a} or Theorem A.

\begin{lemma} \label{obciecie}
{\it Let $X$ be a Banach function space on $I$ such that either the Ces\`aro operator $C$ is bounded 
on $X$ or $X$ is a symmetric space on $[0,1]$ or $X$ is a symmetric space on $[0,\infty)$ with 
$CX[0,\infty) \neq \{ 0 \}$. Then the following embedding
\begin{equation} \label{3}
CX(I)|_{[0,\lambda]} \hookrightarrow L^{1}(I)|_{[0,\lambda]}
\end{equation}
holds for $0 < \lambda < m(I)$, but in general not for $\lambda = 1$ if $I = [0,1]$, and not for 
$\lambda =\infty$ if $I = [0,\infty)$. Moreover, the following equalities hold
\begin{equation} \label{CXrownanie1}
CX[0,\infty)|_{\left[ 0,1\right]} = C(X[0,\infty)|_{[0,1]}) \cap L^{1}[0,1],
\end{equation}
and
\begin{equation} \label{CX-lambda}
CX[0,\infty )|_{\left[ 0,\lambda \right]} = (C(X[0,\infty)|_{[0,1]}))|_{\left[ 0,\lambda \right]} 
\quad \text{for} \quad 0 < \lambda < 1. 
\end{equation}
Finally, if $X$ is a symmetric space on $[0,\infty)$ with $q(X) < \infty $, then 
$C(X[0,\infty)|_{[0,1]}) \neq CX[0,\infty)|_{[0,1]}$ and the space $C(X[0,\infty)|_{[0,1]})$ is 
never a subspace of $CX[0,\infty)|_{[0,1]}$.
}
\end{lemma}

\begin{proof}[Proof of the embedding {\rm (\ref{3})}] Take $f \in CX(I)$ with $\supp(f) \subset [0,\lambda]$, 
where $0 < \lambda < m(I)$. In this case we must ensure that 
$(f_\lambda \colon I \ni x \mapsto \frac{1}{x}\chi_{(\lambda,m(I))}(x)) \in X$. But this immediately 
follows from Theorem D and our assumptions. Keeping this in mind we have the following inequalities
\begin{align*}
\left\Vert f\right\Vert _{CX(I)}:=\left\Vert \frac{1}{x}\int_{0}^{x}\left\vert
f(t)\right\vert \mathrm{d}t\right\Vert _{X(I)}
& =\left\Vert \frac{1}{x} \int_{0}^{x}\left\vert f(t)\right\vert \mathrm{d}t\chi _{(0,\lambda ]}(x)+ 
\frac{1}{x}\int_{0}^{\lambda }\left\vert f(t)\right\vert \mathrm{d}t\chi_{(\lambda ,m(I) )}(x)\right\Vert _{X(I)} \\
& \geq \left\Vert \frac{1}{x}\int_{0}^{\lambda }\left\vert f(t)\right\vert 
\mathrm{d}t\chi _{(\lambda, m(I))}(x)\right\Vert _{X(I)} \\
& =\left\Vert \frac{1}{x}\chi _{(\lambda, m(I))}(x)\right\Vert_{X(I)}\left\Vert f\right\Vert _{L^{1}(I)|_{[0,\lambda ]}}.
\end{align*} 
Therefore, $CX(I)|_{[0,\lambda]} \hookrightarrow L^{1}(I)|_{[0,\lambda]}$. Counterexamples for
embeddings $CX \hookrightarrow L^1$ when either $\lambda = 1$ or $\lambda = \infty$ can be found 
in \cite[Theorem 1 (d)]{AM09}. Note, however, that for example  $Ces_\infty[0,1] \hookrightarrow L^1[0,1]$, 
see \cite[Theorem 1 (d)]{AM09}.	
\smallskip
	
{\it Proof of the equality} {\rm (\ref{CXrownanie1})}. Take $f \in C(X[0,\infty)|_{[0,1]}) \cap L^1[0,1]$.
To prove this equality we will need to know that 
$(f_{\lambda = 1} \colon [0,\infty) \ni x \mapsto \frac{1}{x}\chi_{(1, \infty)}(x)) \in X[0,\infty)$.
Again, our assumptions together with Theorem D ensure that this is indeed the case. Since
\begin{align*}
\left\Vert f\right\Vert _{CX[0,\infty )} & \leq \left\Vert \frac{1}{x} 
\int_{0}^{x}\left\vert f(t)\right\vert \mathrm{d}t\chi _{(0,1]}\left(x\right) \right\Vert _{X[0,\infty )}+\left\Vert \frac{1}{x}
\int_{0}^{1}\left\vert f(t)\right\vert \mathrm{d}t\chi _{\left( 1,\infty \right) }\left( x\right) \right\Vert _{X[0,\infty )} \\
& =\left\Vert f\right\Vert _{C(X[0,\infty)|_{[0,1]})}+\left\Vert \frac{1}{x}\chi _{\left(1,\infty \right) }\left( x\right) 
\right\Vert _{X[0,\infty )}\left\Vert
f\right\Vert _{L^{1}\left[ 0,1\right] }
\end{align*}
it follows that
\begin{equation*}
C(X[0,\infty)|_{[0,1]}) \cap L^{1}\left[ 0,1\right] \hookrightarrow CX[0,\infty)|_{\left[ 0,1\right] }.
\end{equation*}
	
Next, we will show the reverse embedding. Take $f\in CX[0,\infty )$ with $\supp(f)\subset \lbrack 0,1]$.
Again, as above, we have $\frac{1}{x}\chi_{(1, \infty)}(x) \in X$, whence
\begin{align*}
\left\Vert f\right\Vert _{CX[0,\infty )} & \geq \max \left\{ \left\Vert
\left( C\abs{f}\right) \chi _{\left( 0,1\right) }\right\Vert _{X[0,\infty
)},\left\Vert \frac{1}{x}\int_{0}^{1}\left\vert f(t)\right\vert \mathrm{d}t\chi _{\left( 1,\infty
\right) }\left( x\right)\right\Vert_{X[0,\infty )}\right\} \\
& =\max \left \{ \left\Vert f\right\Vert _{C(X[0,\infty)|_{[0,1]})},\left\Vert \frac{1}{x}\chi _{\left( 1,\infty
\right) }\left( x\right) \right\Vert _{X[0,\infty)} \left\Vert f\right\Vert_{L^{1}\left[ 0,1\right] } \right \}.
\end{align*}
In consequence,
\begin{equation*}
CX[0,\infty )|_{\left[ 0,1\right] }\hookrightarrow C(X[0,\infty)|_{[0,1]}) \cap L^{1}\left[ 0,1\right],
\end{equation*}
which proves the equality (\ref{CXrownanie1}).
\smallskip

{\it Proof of the equality} {\rm (\ref{CX-lambda})}. Of course, if $X$ and $Y$ are Banach function 
spaces on $I$, then
\begin{equation} \label{1234}
(X \cap Y)|_{A} \equiv X|_{A} \cap Y|_{A} \quad \text{for every} \quad A \subset I,
\end{equation}
and because $\norm{f}_{(X|_{A})|_{B}} = \norm{f\chi_{B}}_{X|_{A}} = \norm{f\chi_{B}\chi_{A}}_X 
= \norm{f\chi_{A \cap B}}_X = \norm{f}_{X|_{A \cap B}}$, so
\begin{equation*}
(X|_{A})|_{B} \equiv X|_{A \cap B} \quad \text{for every} \quad A, B \subset I.
\end{equation*}
Combining the above equalities with the embedding (\ref{3}) we get
\begin{align*}
CX[0,\infty)|_{[0,\lambda]} & \equiv (CX[0,\infty)|_{[0,1]})|_{[0,\lambda]} \\
& = (C(X[0,\infty)|_{[0,1]}) \cap L^{1}[0,1])|_{[0,\lambda]} \\
& = (C(X[0,\infty)|_{[0,1]}))|_{[0,\lambda]}.
\end{align*}
This gives the equality (\ref{CX-lambda}).
	
Finally, suppose that $q(X) < \infty$. Then as in \cite[Theorem 1 (d)]{AM09} we can show that the function
$$
f(x) = \frac{1}{1-x} \quad \text{for} \quad 0 \leq x < 1,
$$
belongs to the space $C(X[0,\infty)|_{[0,1]})$. In fact, by (\ref{wlozenie}) Theorem B, we have
$L^{1} \cap L^{q}[0,\infty) \overset{A}{\hookrightarrow} X[0,\infty)$ for $q(X) < q <\infty$. Therefore,
using equation (\ref{1234}) we have
$$
L^q[0,1] = L^{1} \cap L^{q}[0,1] = (L^{1} \cap L^{q}[0,\infty))|_{[0,1]} \overset{A}{\hookrightarrow} X[0,\infty)|_{[0,1]}.
$$
Moreover, $\int_0^1(\frac{1}{x}\ln(\frac{1}{1-x}))^q \mathrm{d}x < \infty$ (cf. \cite[Theorem 1 (d), p. 334]{AM09}),
whence 
\begin{align*}
\left\Vert f\right\Vert _{C(X[0,\infty)|_{[0,1]})}^{q} & =\left\Vert C\left\vert f\right\vert
\right\Vert _{X[0,\infty)|_{[0,1]}}^{q} \lesssim \left\Vert C\left\vert f\right\vert
\right\Vert _{L^{q}[0,1]}^{q} \\
& = \int_{0}^{1}(\frac{1}{x}\int_{0}^{x}\frac{\mathrm{d}t}{1-t})^{q}
\mathrm{d}x = \int_{0}^{1}(\frac{1}{x}\ln (\frac{1}{1-x}))^{q}\mathrm{d}x < \infty.
\end{align*}
Of course, $f\notin L^{1}[0,1]$ and this ends the proof of this lemma.
\end{proof}

Let us also note that if $X = L^\infty[0,\infty)$, then $q(X) = \infty$ and $L^\infty[0,\infty)|_{[0,1]} \equiv L^\infty[0,1]$.
Therefore, $C(X[0,\infty)|_{[0,1]}) = Ces_\infty[0,1]$ and, in view of equality (\ref{CXrownanie1}) and embedding
$Ces_\infty[0,1] \hookrightarrow L^1[0,1]$ (see \cite[Theorem 1 (d)]{AM09}), we have
$$CX[0,\infty)|_{[0,1]} = Ces_\infty[0,\infty)|_{[0,1]} = Ces_\infty[0,1] \cap L^1[0,1] = Ces_\infty[0,1].$$
This means that it can happen that $C(X[0,\infty)|_{[0,1]}) = CX[0,\infty)|_{[0,1]}$ and the assumption about
the Boyd index in the last part of the above lemma cannot be omitted.

\section{On a transfer of properties between $X$ and $TX$}

Inclusions and equalities between Ces\`aro spaces $CX$ and Copson spaces $C^*X$ are collected, 
for example, in \cite[Theorem 1]{LM15p} (see also \cite{AM09} and \cite{Be96}). Recall, that if $X$ 
is a Banach function space on $[0,\infty)$ such that both operators $C$ and $C^*$ are bounded on 
$X$, then 
\begin{equation} \label{111}
C^*X \lhook\joinrel\xrightarrow{A} CX \lhook\joinrel\xrightarrow{B} C^*X,
\end{equation}
where $A = \norm{C}_{X \rightarrow X}$ and $B = \norm{C^*}_{X \rightarrow X}$, that is, $CX = C^*X$, 
see \cite[Theorem 1 (iii)]{LM15p}. However, if $I = [0,1]$, then the situation is a bit more complicated. 
More precisely, if $X$ is a Banach function space on $[0,1]$ such that the Ces\`aro and Copson 
operators are bounded on $X$ and $L^\infty[0,1] \hookrightarrow X \hookrightarrow L^1[0,1]$
(for example, if $X$ is a symmetric space), then
\begin{equation} \label{1111}
CX[0,1] \cap L^1[0,1] = C^*X[0,1],
\end{equation}
see \cite[Theorem 1 (vi) and (vii)]{LM15p}. Therefore, at least when we consider Banach function spaces 
on $[0,\infty)$ such that both operators $C$ and $C^*$ are bounded on $X$, all results regarding 
the isomorphic structure of the Ces\`aro function spaces ``transfer" almost trivially to the case of 
the Copson function spaces and vice versa. However, it may not be the case if $I = [0,1]$. After all, 
we will prove the following

\begin{corollary} \label{Copson a Cesaro}
{\it Let $X$ be a symmetric space such that both operators $C$ and $C^*$ are bounded on $X$.
Then the Copson space $C^*X$ is order continuous if and only if $X$ is order continuous.}
\end{corollary}
\begin{proof}
First of all, $X$ is order continuous if and only if $CX$ is also, whenever $X$ is a symmetric space 
such that the Ces\`aro operator $C$ is bounded on $X$, see \cite[Theorem 3]{KT17}). Thus, 
according to the equality (\ref{111}) and the discussion preceding this result, there is noting to 
prove when $I = [0,\infty)$. Therefore, let us focus on the case when $I = [0,1]$.
	
Suppose that $X \in (OC)$. If $X \in (OC)$, then $CX \in (OC)$, see \cite[Lemma 1 (a)]{LM15p}. 
Of course, $L^1 \in (OC)$, so $CX[0,1] \cap L^1[0,1] \in (OC)$. Using the equality (\ref{1111}) we 
see immediately that $C^*X[0,1] \in (OC)$.
	
To prove the reverse implication, assume that $X \notin (OC)$. Because $X$ is a symmetric space, 
we can find an element $f_0 \in X$ such that $f_0 \notin X_a$ but $f_0 \in L^1[0,1]$. Without loss 
of generality, we can also assume that $f_0 = f_0^*$, see \cite[Lemma 2.6]{CKP14}. From the 
boundedness of the Ces\`aro operator $C$ it follows that $C(f_0) \in X$. Moreover, the element 
$f_0$ is a nonincreasing function and $X_a$ is an order ideal of $X$ (see \cite[Theorem 3.8, p. 16]{BS88}),
so $C(f_0) \geq f_0$ and also $C(f_0) \notin X_a$. Therefore, $f_0 \notin C(X_a) = (CX)_a$, 
see \cite[Theorem 16]{KT17}. In summary, $f_0 \in (CX\setminus (CX)_a) \cap L^1[0,1]$. But, in view 
of the equality (\ref{1111}), this means that $f_0 \in C^*X\setminus (C^*X)_a$, i.e., $C^*X \notin (OC)$.
\end{proof}

It may happen that applying the construction $X \mapsto TX$, where $T = C$ or $T = C^*$, we lose 
some information about the original space $X$. We will give rather general example of this kind.

The idea behind the next lemma is simple. Every Ces\`aro and Copson function space contain ``in the 
middle" an isomorphic copy of $L^1[0,1]$ (cf. Lemma \ref{complementedL1} and Lemma 
\ref{complemented L1 in C*X}). Therefore, up to equivalence of norms, we can change the space $X$
``in the middle" (cf. the equality (\ref{98})) and still get the Ces\`aro or Copson function space 
equal to the original one (cf. the equality (\ref{CX = CZ})).

\begin{lemma} \label{CX zjada srodek}
{ \it Let $T = C$ or $T = C^*$. Define the Banach function space $Z = Z[0,1]$ as
\begin{equation} \label{98}
Z[0,1] := X|_{[0,a]} \oplus Y|_{[a,b]} \oplus X|_{[b,1]} \quad \text{for} \quad 0 < a < b < 1,
\end{equation}
where $X$ and $Y$ are Banach function spaces on $[0,1]$ such that 
$L^\infty[0,1] \hookrightarrow Y \hookrightarrow X$. Then
\begin{equation} \label{CX = CZ}
TZ[0,1] = TX[0,1].
\end{equation}
}
\end{lemma}
\begin{proof}[Proof of the embeddings $CZ \hookrightarrow CX$ and $C^*X \hookrightarrow C^*Z$]
In the proof of this part we need only the assumption that $Y \hookrightarrow X$. Indeed, then
\begin{equation*}
Z[0,1] = X|_{[0,a]} \oplus Y|_{[a,b]} \oplus X|_{[b,1]} \hookrightarrow X|_{[0,a]} \oplus X|_{[a,b]} 
\oplus X|_{[b,1]} = X[0,1].
\end{equation*}
Note that just by the definition if $E$ and $F$ are Banach function spaces on $I$ and 
$E \hookrightarrow F$, then $TE \hookrightarrow TF$. Consequently, $TZ \hookrightarrow TX$.
\smallskip
	
{\it Proof of the embedding $CX \hookrightarrow CZ$}. First, observe that if $W$ is a Banach function 
space on $[0,1]$ such that $\supp \left(CW\right) =[0,1]$, then 
\begin{equation}
CW[0,1]|_{[a,b]}=L^{1}[0,1]|_{[a,b]}=L^{1}[a,b]\quad \text{for}\quad 0<a<b<1,
\label{101}
\end{equation}
due to Lemma \ref{complementedL1}. Take $f \in CX$ and denote $f_{1} = f\chi _{\lbrack 0,a)}$, 
$f_{2} = f\chi _{\lbrack a,b)}$ and $f_{3} = f\chi _{\lbrack b,1]}$. We need to show that 
$f_{1}, f_{2}, f_{3}\in CZ.$ Note that $\left( C\left\vert f_{1}\right\vert \right) \chi _{\lbrack 0,a)} 
=\left( C\left\vert f\right\vert \right) \chi_{\lbrack 0,a)}\in X.$ Moreover, $\left( C\left\vert f_{1}\right\vert
\right) \chi _{\lbrack a,b)}\in L^{\infty }\left[ 0,1 \right]$, so 
$\left( C\left\vert f_{1}\right\vert \right) \chi _{\lbrack a,b)}\in Y$, because
$L^{\infty}[0,1] \hookrightarrow Y$. Next, observe that 
$\left( C\left\vert f_{1}\right\vert \right) \chi _{\lbrack b,1)}\leq \left( C\left\vert f\right\vert \right)
\chi _{\lbrack b,1)}\in X$. Thus $C\left\vert f_{1}\right\vert \in Z$. Since 
$L^{\infty}[0,1] \hookrightarrow Y\hookrightarrow X$, so $L^\infty[0,1] \hookrightarrow Z$ and 
$\supp \left( CX\right) = \supp\left( CZ\right) = [0,1]$. In consequence, by equality (\ref{101}), we have
$f_{2}\in CX|_{[a,b]} = L^{1}[0,1]|_{[a,b]} = CZ|_{[a,b]}$. Finally, 
$C\left\vert f_{3}\right\vert = C\left(\left\vert f\right\vert \chi _{[b,1]}\right) \leq C\left( \left\vert
f\right\vert \right) \chi _{\lbrack b,1]}\in X|_{[b,1]} = Z|_{[b,1]}$. Consequently, $f \in CZ$, but this 
means that $CX\hookrightarrow CZ$.
\smallskip
	
{\it Proof of the embedding $C^{\ast }X\hookrightarrow C^{\ast }Z$}. Let $f\in C^{\ast }X$ and set 
$f_{1}=f\chi _{\lbrack 0,a)}$, $f_{2}=f\chi _{\lbrack a,b)}$ and $f_{3}=f\chi _{\lbrack b,1]}$. Since
$C^{\ast }\left\vert f_{1}\right\vert \leq C^{\ast }\left\vert f\right\vert \in X$, it follows that
$\left( C^{\ast }\left\vert f_{1}\right\vert \right) \chi _{\lbrack 0,a)}\in Z$. Moreover, 
$\left( C^{\ast }\left\vert f_{1}\right\vert \right) \chi_{\lbrack a,1]}\equiv 0,$ whence $f_{1}\in CZ$. 
Just as above we conclude that $\supp\left( C^{\ast }X\right) =\supp\left( C^{\ast }Z\right) =[0,1]$. 
Thus, by Lemma \ref{complemented L1 in C*X}, $f_{2}\in C^{\ast }X|_{[a,b]}=L^{1} \left[ a,b\right] = 
C^{\ast }Z|_{[a,b]}$. Moreover, $C^{\ast }\left\vert f_{3}\right\vert \leq C^{\ast }\left\vert f\right\vert \in X$, 
so $\left( C^{\ast }\left\vert f_{3}\right\vert \right) \chi _{\lbrack 0,a)} \in Z$ and 
$\left( C^{\ast }\left\vert f_{3}\right\vert \right) \chi _{\lbrack b,1]}\in Z$. Finally, note that 
$C^{\ast }\left\vert f\right\vert \in X$ and the function $C^{\ast }\left\vert f\right\vert$ is nonincreasing,
thus $\left( C^{\ast }\left\vert f_{3}\right\vert \right) \chi _{\lbrack a,b)}\left( x\right) =
\int_{b}^{1}\frac{\left\vert f\left( t\right) \right\vert }{t}dt=C^{\ast }\left\vert f\right\vert \left( b\right) <\infty $
for each $x \in \lbrack a,b)$. Therefore, we obtain that 
$\left( C^{\ast }\left\vert f_{3}\right\vert \right) \chi_{\lbrack a,b)}\in L^{\infty }[0,1]|_{[a,b]} \hookrightarrow Y|_{[a,b]}$ 
and this ends the proof.
\end{proof}

The above lemma can be viewed as an abstract version of Example 1 from \cite{LM15p}. In particular, 
if $X$ is a symmetric space on $[0,1]$, then $L^\infty[0,1] \hookrightarrow X \hookrightarrow L^1[0,1]$ and
$$
T(L^1[0,a] \oplus X|_{[a,b]} \oplus L^1[b,1]) \simeq L^1[0,1],
$$
but
$$
T(X|_{[0,a]} \oplus L^\infty[a,b] \oplus X|_{[b,1]}) = TX[0,1].
$$
In connection with the above lemma, the following simple observation is worth noting.

\begin{lemma} \label{X = Y, to CX = CY}
{\it Let $T = C$ or $T = C^*$. Assume that $X$ and $Y$ are symmetric spaces on $I$ such that 
\begin{enumerate}
\item [(i)] the Ces\`aro operator $C$ is bounded on $X$ and $Y$, respectively, if $T = C$,
\item [(ii)] both operators $C$ and $C^*$ are bounded on $X$ and $Y$, respectively, if $T = C^*$.
\end{enumerate}
Then $X = Y$ if and only if $TX = TY$.}
\end{lemma}
\begin{proof}
If $X = Y$, then $X \hookrightarrow Y$ and $Y \hookrightarrow X$, so $TX \hookrightarrow TY$ and 
$TY \hookrightarrow TX$. Thus, $TX = TY$.
	
Suppose that $T = C$. If $CX \hookrightarrow CY$, then thanks to boundedness of the Ces\`aro operator 
and the symmetry of $X$, we have
$$ 
\norm{f}_Y = \norm{f^*}_Y \leq \norm{C(f^*)}_Y \lesssim \norm{C(f^*)}_X 
\leq \norm{C}_{X \rightarrow X} \norm{f^*}_X \lesssim \norm{f}_X,
$$
that is, $X \hookrightarrow Y$. Changing the roles of $X$ and $Y$ we can show the reverse embedding 
$Y \hookrightarrow X$.
		
Let $T = C^*$ and assume that $C^*X \hookrightarrow C^*Y$. If $I = [0,\infty)$, then $C^*X = CX$ 
and there is nothing to prove, see the equality (\ref{111}). On the other hand, if $I = [0,1]$, then 
$C^*X = CX \cap L^1$, see the equality (\ref{1111}). Since and $X \hookrightarrow C^*X$, it follows 
that
$$
X \hookrightarrow C^*X \hookrightarrow C^*Y = CY \cap L^1 \hookrightarrow CY.
$$
Moreover, $\norm{f}_Y = \norm{f^{*}}_Y \leq \norm{C(f^*)}_Y$ and consequently
$$
\norm{f}_Y \leq \norm{C(f^*)}_Y = \norm{f^*}_{CY} \lesssim \norm{f^*}_X = \norm{f}_X,
$$
that is, $X \hookrightarrow Y$. 
\end{proof}

If we try to reformulate Lemma \ref{X = Y, to CX = CY} using isomorphism instead of the ``equalities", 
then this result is not longer true. For example, if $X = L^\infty$ and $Y = L^\infty(t)$, then 
$CY \equiv L^1$ and $X \simeq Y$ but of course $L^1$ is not isomorphic to $Ces_\infty$ (because 
$L^1$ is separable and $Ces_\infty$ is not). On the other hand, if $X = L^1[0,1]$ and 
$Y = L^\infty(t)[0,1]$, then $CX \simeq L^1[0,1]$, $CY \equiv L^1[0,1]$ and $CX \simeq CY$ but 
$X$ is not isomorphic to $Y$.

Generally, some isomorphic as well as isometric properties inherit well from $X$ to the Ces\`aro space 
$CX$ (for example, order continuity \cite{LM15p}, Fatou property \cite[Theorem 1 (d)]{LM15a} and 
rotundity \cite{Ki-Kol-rot}). However, there are properties, like reflexivity, which Ces\`aro 
function spaces never have (cf. Corollary \ref{CX-contains compl of l1}). In other words, certain properties never transfer 
from $X$ to $CX$. Below we present next two properties of this kind.

\begin{corollary} \label{not dual and not RNP}
{\it Let $T = C$ or $T = C^*$. Suppose that $X$ is an order continuous Banach 
function space on $I$ with $TX\neq \{0\}$. Then
\begin{enumerate}
\item [(i)] $TX$ is not isomorphic to a dual space.
\item [(ii)] $TX$ does not have the Radon--Nikodym property.
\end{enumerate}
	}
\end{corollary}
\begin{proof}
Part (i) of the proof is the same as in \cite[Theorem 3]{AM13b}. We will give the details for 
a sake of completeness.
	
(i) We argue by contradiction. Suppose that $TX$ is isomorphic to a dual space, i.e. there exist 
a Banach function space $Y$ with $(TX, \norm{\cdot}_{TX}) \simeq Y^*$. By Theorem \ref{general renorming} 
we can find an equivalent norm, say $\norm{\cdot}'$, on the space $TX$ such that $(TX, \norm{\cdot}')$ 
contains a closed subspace isometric to $L^{1}[0,1]$. Of course, $(TX, \norm{\cdot}') \simeq Y^*$. 
It follows from the definition that if $X \in (OC)$ then $TX \in (OC)$, cf. also \cite[Lemma 1 (a)]{LM15p}.
Thus, our assumptions show that $(TX, \norm{\cdot}') \in (OC)$. Now, we can apply the well known 
fact that a Banach function space $X$ over the measure $\mu$ is separable if and only if it is order 
continuous and the measure $\mu $ is separable \cite[Theorem 5.5]{BS88} to conclude that 
$(TX, \norm{\cdot}')$ is also separable. Applying the Bessaga--Pe\l {}czy\'{n}ski result, see \cite{BesPel}, 
it follows that $(TX, \norm{\cdot}')$ has the Krein--Milman property. Therefore, every closed bounded 
set in $(TX, \norm{\cdot}')$ is a closed convex hull of its extreme points. On the other hand, the closed 
unit ball in $L^{1}[0,1]$ has no extreme points. This contradiction ends the proof.
	
(ii) This case follows from the Talagrand theorem, see \cite[Corollary 5.4.21]{MN91}, which states that 
a separable Banach lattice is isomorphic to a dual Banach lattice if and only if it has the Radon--Nikodym 
property.
\end{proof}

The above Corollary has been proved for $Ces_{p}$-spaces in \cite[Theorem 3]{AM13b} and 
in \cite[Corollaries 5.1 and 5.5]{KK12} (using duality arguments). Moreover, part (i) of the above result 
for the Ces\`aro function spaces $CX$, where $X$ is an order continuous symmetric space such 
that the Ces\'aro operator is bounded on $X$, is included in \cite[Proposition 5.3]{ALM17}. 
Interestingly, it may happen that the Ces\`aro function space $CX$ for a nonseparable space $X$ 
is isomorphic to a dual space. It was proved in \cite[the equality (2.9) and Theorem 5.1]{ALM17} that 
\begin{equation*}
(\widetilde{\mathit{\ell ^{1}}})^{\ast }=(\widetilde{\mathit{\ell ^{1}}}
)^{\prime }=ces_{\infty }\simeq Ces_{\infty },
\end{equation*}
which means that $Ces_{\infty }$ is isomorphic to a dual space.

The question when a given property ``transfers'' also in the opposite direction, i.e. from $CX$ to $X$, 
imposes itself. However, as the next theorem will show, in the class of Banach function spaces the 
answer is basically always negative. Before we formulate this result we need the following definition.

Let $X$ be a Banach function space with the property $P$ (in short, $X \in (P)$). We will say that 
the property $P$ is {\it good for the Ces\`aro construction} if $P$ is invariant under equivalent 
renormings (that is, if $(X,\norm{\cdot}) \in (P)$ and $\norm{\cdot}'$ is an equivalent norm on  $X$, 
then also $(X,\norm{\cdot}') \in (P)$) and we can find two nontrivial symmetric spaces $X$ and 
$Y$ on $[0,1]$, which satisfy the following conditions:
\begin{enumerate}
\item [(G1)] $X \in (P)$ but $Y \notin (P)$,
\item [(G2)] $X \cap Y \notin (P)$,
\item [(G3)] $X|_A \in (P)$ for every $\emptyset \neq A \subset [0,1]$,
\item [(G4)] $CX \in (P)$.
\end{enumerate}
Similarly, we will say that the property $P$ is {\it good for the Copson construction} replacing 
the condition (G4) with ``$C^*X \in (P)$" in the above definition. Moreover, the property $P$ 
is {\it good} if is good for the Ces\`aro and Copson construction. The definitions given do not 
look particulary restrictive, however, we will give some examples of such properties.

\begin{example}
(a) Let us start with the fact that order continuity is good property. Take $P = OC$. First, note that 
if $X \in (OC)$, then $X|_A \in (OC)$ for every $\emptyset \neq A \subset [0,1]$. Moreover, if 
$X \in (OC)$, then $TX \in (OC)$, see \cite[Lemma 1 (a)]{LM15p} and Corollary \ref{Copson 
a Cesaro}. It remains to find two nontrivial symmetric spaces on $[0,1]$ with properties (G1) 
and (G2). For example, let $X = L^1[0,1]$ and $Y$ be any symmetric space on $[0,1]$ with 
$Y \neq Y_a$ or let $Y = L^\infty[0,1]$ and $X$ be any order continuous symmetric space 
on $[0,1]$. It is clear that in both cases $X \cap Y \notin (OC)$.	
\smallskip
	
(b) We will show that the Dunford--Pettis property (for short, $DPP$) is good property (see 
\cite[p. 115]{AK06} for the definiton and \cite{Di80} for related results). Kami\'nska--Masty\l {}o 
proved in \cite{KM00} that there are exactly two nonisomorphic symmetric spaces on $[0,1]$ 
with the Dunford--Pettis property, namely $L^1[0,1]$ and $L^\infty[0,1]$. Therefore, if 
$X = L^1[0,1]$ and $Y$ is a reflexive symmetric space on $[0,1]$, then $X \in (DPP)$, 
$Y \notin (DPP)$, $X \cap Y = Y \notin (DPP)$ and $CL^1[0,1] = Ces_1[0,1] \simeq L^1[0,1] \in (DPP)$.
Moreover, $C^*L^1[0,1] = Cop_1[0,1] \simeq L^1[0,1] \in (DPP)$. Finally, a complemented 
subspaces of spaces with the Dunford--Pettis property also have it, so the condition (G3) is 
satisfied in an obvious way.
\smallskip	
	
(c) Let $p \geq 1$ and suppose that $X$ is a Banach function space on $I$ which is $p$-concave 
with constant $L \geq 1$ (see \cite[pp. 45--46]{LT79} for the definition) and such that the Ces\`aro 
operator $C$ is bounded on $X$. First, we will show that then also the space $CX$ is $p$-concave 
with constant $L$. Recall, that the space $L^1(I)|_{[0,x]}$ for $0 < x \in I$ is 1-convex with constant 
1 and $p$-concave with constant 1, that is,
\begin{align*}
\left(\sum_{k=1}^n \norm{f_k}_{L^1(I)|_{[0,x]}}^p \right)^{1/p}
& = \left( \sum_{k=1}^n \left(\int_0^x \abs{f_k(t)} \mathrm{d}t \right)^{p} \right)^{1/p} \\
& \leq \int_0^x \left(\sum_{k=1}^n \abs{f_k(t)}^p \right)^{1/p} \mathrm{d}t 
= \norm{\left(\sum_{k=1}^n \abs{f_k}^p \right)^{1/p}}_{L^1(I)|_{[0,x]}},
\end{align*}
for every $0 < x \in I$, see \cite[Proposition 1.d.5]{LT79}, \cite[Theorem 4.3]{Ma03} and 
the second part of the proof in \cite[Theorem 4]{AM14}. Therefore, we see immediately that
\begin{align*}
\left(\sum_{k=1}^n \left(C\abs{f_k} \right)^p \right)^{1/p}
& = \left(\sum_{k=1}^n \left(\frac{1}{x}\int_0^x \abs{f_k(t)} \mathrm{d}t \right)^{1/p} \right)^{1/p} \\
& \leq \frac{1}{x} \int_0^x \left(\sum_{k=1}^n \abs{f_k(t)}^p \right)^{1/p} \mathrm{d}t
= C \left(\sum_{k=1}^n \abs{f_k}^p \right)^{1/p}.
\end{align*}
Using the above inequality and $p$-concavity of the space $X$, we have
\begin{align*}
\left(\sum_{k=1}^n \norm{f_k}_{CX}^p \right)^{1/p} & = \left(\sum_{k=1}^n \norm{C\abs{f_k}}_{X}^p \right)^{1/p} \\
& \leq L \norm{\left(\sum_{k=1}^n \left(C\abs{f_k} \right)^p \right)^{1/p}}_{X} \\
& \leq L \norm{ C \left(\sum_{k=1}^n \abs{f_k}^p \right)^{1/p}}_{X}
= L \norm{\left(\sum_{k=1}^n \abs{f_k}^p \right)^{1/p}}_{CX},
\end{align*}
for all $f_1, f_2, ..., f_n \in CX$. Consequently, also the space $CX$ is $p$-concave with the 
same constant as for $X$. Taking, for example, $X = L^p[0,1]$ for $1 < p < \infty$ and $Y = L^\infty[0,1]$ 
we conclude that $p$-concavity is good property for the Ces\`aro construction.
	
To show that $p$-concavity is a good property also for the Copson construction we will prove first 
that if $X$ and $Y$ are $p$-concave Banach function spaces on $I$, then the space $X \cap Y$ is 
$p$-concave as well. Let $Z = X \cap Y$ and $\norm{f}_Z = \max\{\norm{f}_X, \norm{f}_Y \}$. 
Take $f_1, f_2, ..., f_n \in Z$ and denote by $L_X, L_Y > 0$ the constants of $p$-concavity of $X$ 
and $Y$, respectively (cf. \cite{LT79}). We have
\begin{equation*}
\norm{\left( \sum_{k=1}^{n} \abs{f_k}^p \right)^{1/p}}_Z \geq \norm{\left( \sum_{k=1}^{n} \abs{f_k}^p \right)^{1/p}}_X
\geq \frac{1}{L_X} \left( \sum_{k=1}^n \norm{f_k}_X^p \right)^{1/p},
\end{equation*}
and
\begin{equation}
\norm{\left( \sum_{k=1}^{n} \abs{f_k}^p \right)^{1/p}}_Z \geq \norm{\left( \sum_{k=1}^{n} \abs{f_k}^p \right)^{1/p}}_Y
\geq \frac{1}{L_Y} \left( \sum_{k=1}^n \norm{f_k}_Y^p \right)^{1/p}.
\end{equation}
Therefore, setting $L = \max\{L_X, L_Y \}$, we have
\begin{equation*}
\norm{\left( \sum_{k=1}^{n} \abs{f_k}^p \right)^{1/p}}_Z
\geq \frac{1}{2L} \left( \left( \sum_{k=1}^n \norm{f_k}_X^p \right)^{1/p} + \left( \sum_{k=1}^n \norm{f_k}_Y^p \right)^{1/p} \right).
\end{equation*}
From the triangle inequality for the space $\ell^p$, we obtain immediately that
\begin{align*}
\norm{\left( \sum_{k=1}^{n} \abs{f_k}^p \right)^{1/p}}_Z
& \geq \frac{1}{2L} \left( \sum_{k=1}^n \left( \norm{f_k}_X + \norm{f}_Y \right)^p \right)^{1/p} \\
& \geq \frac{1}{2L} \left( \sum_{k=1}^n \left( \max\{\norm{f}_X, \norm{f}_Y \} \right)^p \right)^{1/p}
= \frac{1}{2L} \left( \sum_{k=1}^n \norm{f_k}_Z^p \right)^{1/p}.
\end{align*}
But this means that the space $Z$ is actually $p$-concave and the claim follows.
	
Now, it is clear that if $X$ is $p$-concave symmetric space such that both operators $C$ and $C^*$ 
are bounded on $X$, then the Copson space $C^{*}X$ is also $p$-concave. Indeed, suppose that 
$I = [0,1]$, because if $I = [0,\infty)$ there is  noting to prove, cf. the equality (\ref{111}). Since 
$L^1[0,1]$ is $p$-concave with constant 1, it follows from the first part, that the space $CX[0,1]$ 
is also $p$-concave. Therefore, the space $CX[0,1] \cap L^1[0,1]$ is $p$-concave as well.
But in view of the equality (\ref{1111}) this means that also the Copson space $C^*X[0,1]$ is 
$p$-concave.
\smallskip		
	
(d) Let $\Upsilon(X) := \{ p > 1 \colon \text{$X$ contain an isomorphic copy of $\ell^p$} \}$. Then
$$
\Upsilon(Ces_p[0,1]) = \Upsilon(L^p[0,1]) \cup \Upsilon(L^1[0,1]) \quad \text{for} \quad 1 \leq p < \infty,
$$
see \cite[Theorem 10]{AM09} and \cite[Theorem 5.5 and Fig. 1--2]{AM14}. Now, if we set $p > 1$ 
and define a property $P$ to means that the space $X$ contains an isomorphic copy of $\ell^p$, then 
$P$ is also good property for the Ces\`aro construction.
\end{example}

\begin{theorem} \label{transfer}
{\it Let $T = C$ or $T = C^*$. Suppose that $X$ is a Banach function space on $I$ such that $TX$ 
is nontrivial and let $P$ be a good property. The implication: if $TX \in (P)$ then $X \in (P)$, does 
not hold in general.}
\end{theorem}
\begin{proof}
We consider only the case when $P$ is a good property for the Ces\`aro construction, because the 
proof in the second case is the same.
	
Since $P$ is a good property, it follows that there are nontrivial symmetric spaces $X$ and $Y$ on $[0,1]$ 
such that $X \in (P)$ and $Y \notin (P)$ (by property (G1)). Note that we only have two possibilities either 
$Y \hookrightarrow X$ or $X \not\hookrightarrow Y$ and $Y \not\hookrightarrow X$. Indeed, if 
$X \hookrightarrow Y$, then $X \cap Y = X$ and $X \cap Y \notin (P)$ via property (G2) which is impossible. 
Therefore, we will consider two situations.
	
($Y \hookrightarrow X$). Observe, that all assumptions from Lemma \ref{CX zjada srodek} are fulfilled.
Choose $0 < a < b < 1$ and put $Z[0,1] := X[0,a]\oplus Y[a,b]\oplus X[b,1]$. Then $Z|_{[a,b]} \equiv Y$ 
and $Y \notin (P)$, so $Z \notin (P)$ by property (G3). On the other hand, using property (G4), we 
conclude that $CX \in (P)$. But it follows from equality (\ref{CX = CZ}) that $CZ = CX$, thus 
$CZ \neq \{ 0 \}$ and $CZ \in (P)$. In summary, we have shown that there exist a Banach function 
space $Z$ on $[0,1]$ with $Z \notin (P)$, $CZ \neq \{ 0 \}$ and $CZ \in (P)$.
	
($X \not\hookrightarrow Y$ {\it and} $Y \not\hookrightarrow X$). Let us note that
$$
L^\infty \hookrightarrow X \cap Y \hookrightarrow X,
$$
where the first inclusion follows from \cite[Theorem 6.6, p. 77]{BS88} (clearly, $X \cap Y$ is a symmetric 
space). Consequently, $X \cap Y \neq \{ 0 \}$ and $X \cap Y \notin (P)$ by property (G2). 
Moreover, $X \cap Y \hookrightarrow X$ and so the assumptions of Lemma \ref{CX zjada srodek} 
are satisfied also in this case for the spaces $X \cap Y$ and $X$. Now we can continue like in 
the previous case but instead of the space $Y$ we take $X \cap Y$.
\end{proof}

It seems interesting that when we restrict the class of spaces under consideration to class of 
symmetric spaces, then it may happen that the above theorem is not true. For example, it was 
proved by Kiwerski--Tomaszewski \cite[Theorem 3]{KT17} that a space $X$ is order continuous 
if and only if $CX$ is also order continuous, whenever $X$ is a symmetric space such that the 
Ces\`aro operator is bounded on $X$ (see also Corollary \ref{Copson a Cesaro}).
Moreover, if $I = [0,\infty)$ and we allow the situtation that $CX = \{ 0 \}$, then it is easy to see that
$C(L^1 \cap L^\infty) = \{ 0 \} \in (OC)$ but $L^1 \cap L^\infty \notin (OC)$.

The results in the next section show that also fixed point properties do not transfer from $X$ into 
the Ces\`aro or Copson function spaces. 

\section{Applications to the metric fixed point theory} \label{sekcja FPP}

A Banach space $X = (X, \norm{\cdot}_X)$ has the \textit{fixed point property} ($X \in (FPP)$ for short)
if every nonexpansive mapping $T \colon K \rightarrow K$, that is, the mapping satisfying
$$
\norm{T(x) - T(y)} \leq \norm{x-y} \quad \text{for all} \quad x,y \in K,
$$
on every nonempty, closed, bounded and convex subset $K$ of $X$, has a fixed point, i.e., there 
exist a point $x_0 \in K$ such that $T(x_0) = x_0$. If the same holds for every nonempty, weakly 
compact and convex subset $K$ of $X$, we say that this space has the \textit{weak fixed point 
property} (we write $X \in (wFPP)$). Of course, if the space $X$ has the fixed point property,
then $X$ has the weak fixed point property and both properties are equivalent in the class of 
reflexive spaces. The spaces $c_0$, $\ell^1$, $L^1[0,1]$, $L^\infty[0,1]$, $L^{p,1}[0,\infty)$ 
and $C[0,1]$ fail the fixed point property and the spaces $\ell^\infty$, $c_0(\Gamma)$ and 
$\ell^1(\Gamma)$, for $\Gamma$ uncountable, cannot be even renormed to have the fixed 
point property, see Theorem 2, Corollary 3 and remark after Proposition 7 in \cite{DLT96}. 
However, $c_0$ and $\ell^1$ have the weak fixed point property but $L^1[0,1] \notin (wFPP)$, 
as it was proved by Alspach \cite{Al81}.

We are now ready to prove the main result of this section.

\begin{theorem} \label{FPPCesaro} \label{main theorem}
{\it Let $T = C$ or $T = C^*$. If $X$ is a Banach function space on $I$ such that 
$TX \neq \{ 0 \}$, then the space $TX$ fails to have the fixed point property. 
Moreover,
\begin{enumerate}
\item[(i)] the space $\left( TX\right) ^{\ast }$ cannot be renormed to have
the fixed point property,

\item[(ii)] the space $\left( TX\right) ^{\ast }$ fails to have the weak
fixed point property.
\end{enumerate}
}
\end{theorem}
\begin{proof}
The first claim that $CX\notin (FPP)$ follows immediately from Theorem \ref{AIC-l1} and the
Dowling--Lennard result from \cite{DL97}, which states that a Banach space which contains an
asymptotically isometric copy of $\ell ^{1}$ fails to have the fixed point property (see also 
\cite[Theorem 2.3 and Corollary 2.11]{DLT01}). For the Copson space $C^*X$ we apply 
Theorem \ref{AIC-l1 in C^*X}, respectively.

(i) It is known that if a\ Banach space $X$ contains complemented copy of $\ell^{1}$, then 
$X^{\ast }$ cannot be renormed to have the fixed point property (see \cite[Corollary 4]{DLT96}). 
Thus we should apply only Corollary \ref{CX-contains compl of l1}.

(ii) Recall the Dilworth--Girardi--Hagler result \cite[Theorem 2]{DGH00} which states that 
a Banach space $X$ contains an asymptotically isometric copy of $\ell ^{1}$ if and only if 
the dual space $X^{\ast }$ contains an isometric copy of $L^{1}[0,1].$ Combining 
Theorem \ref{as kopia l1} with the Dilworth--Girardi--Hagler result we obtain that the 
space $\left( CX\right)^{\ast }$ contains an isometric copy of $L^{1}[0,1]$. In view of Alspach
result from \cite{Al81} this means that $\left( CX\right) ^{\ast }\notin (wFPP)$. Again, in the 
case of the Copson space $C^*X$ we simply use Theorem \ref{AIC-l1 in C^*X}.
\end{proof}

By the Alspach result \cite{Al81} and our Theorem \ref{general renorming} we easily obtain 
the following corollary.

\begin{corollary}
{\it Let $T = C$ or $T = C^*$. Assume that $X$ is a Banach function space on $I$ with 
$TX\neq \{0\}$. Then there is an equivalent norm on the space $TX$ for which $TX$ 
fails the weak fixed point property.}
\end{corollary}

In the next remark we collect some known results concerning the (weak) fixed point property
and copies of $\ell^\infty$.

\begin{remark}\label{izometryczna kopia ell-nieskonczonosc a FPP}
{\it Let $X$ be a Banach space.
\begin{enumerate}
\item [(i)] If $X$ contains an isomorphic copy of $\ell^\infty$, then $X$ cannot be renormed 
to have the fixed point property,
\item [(ii)] If $X$ contains an isometric copy of $\ell^\infty$, then $X$ fails to have the weak 
fixed point property.
\end{enumerate}
}
\end{remark}
\begin{proof}
(i) It follows from Pe\l czy\'{n}ski result \cite{Pe68} that a separable Banach space $X$ contains 
an isomorphic copy of $\ell^1$ if and only if $X^*$ contains an isomorphic copy of 
$\ell^1(\Gamma)$ for some uncountable set $\Gamma$. In particular, $(\ell^1)^* = \ell^\infty$ 
and consequently $\ell^\infty$ contains an isomorphic copy of $\ell^1(\Gamma)$. Moreover, 
by Dowling--Lennard--Turett result \cite[Theorem 1]{DLT96} any renorming of the space 
$\ell^1(\Gamma)$ contains an asymptotically isometric copy of $\ell^1$. But a Banach space 
which contains an asymptotically isometric copy of $\ell^1$ fails the fixed point property \cite{DLT01}. 
Therefore, the space $X$ fails the fixed point property as well.
	
(ii) A classical result is that $\ell^\infty$ is the universal space for all separable Banach spaces, 
i.e., every separable Banach space $X$ can be isometrically embedded into $\ell^\infty$ (just 
take a dense subset $\{x_n \colon n \in \mathbb{N}\} \subset S(X)$ with $x_n^*x_n = 1$ for all 
$n \in \mathbb{N}$, where $\{ x_n^* \colon n \in \mathbb{N} \} \subset S(X^*)$, and put 
$T \colon X \ni x \mapsto (x_n^* x)_{n=1}^\infty \in \ell^\infty$). Therefore, in particular, $\ell^\infty$ 
contains an isometric copy of $L^1[0,1]$. Again, by the Alspach result \cite{Al81}, $X \notin (wFPP)$.
\end{proof}

Lozanovski{\u \i} proved in \cite{Lo69} that a Banach function space $X$ is order continuous if and only if it contains no isomorphic 
copy of $\ell^\infty$. Moreover, if $X$ is a symmetric space and $C$ is bounded 
on $X$, then the  space $CX$ is order continuous if and only if $X$ is also (see \cite{KT17}). 
Consequently, by the above Remark \ref{izometryczna kopia ell-nieskonczonosc a FPP} and 
Corollary \ref{Copson a Cesaro}, we have the following

\begin{corollary} \label{TX bez przenormowania do FPP}
{\it Let $T = C$ or $T = C^*$. If $X$ is a symmetric space on $I$ such that $C$ and $C^*$ are bounded 
on $X$ and $X$ is not order continuous, then $TX$ cannot be renormed to have the fixed point property}.
\end{corollary}

\begin{proposition} \label{Proposition wFPP}
{\it The spaces $Ces_{1}[0,1]$, $Ces_\infty$, $Cop_1$, $Cop_\infty$ and nontrivial Tandori function 
spaces $\widetilde{\mathit{X}}$ fail to have the weak fixed point property.}
\end{proposition}
\begin{proof}
Since $Ces_{1}[0,1]\equiv L^{1}(\ln (1/t))[0,1]$ (see Example \ref{Cesaro i Copson dla L1 i L8} (b)), 
$L^{1}(\ln (1/t))[0,1]$ is isometric to $L^{1}[0,1]$ and $L^{1}[0,1]\notin (wFPP)$ by the Alspach 
result \cite{Al81}, so $Ces_{1}[0,1]\notin (wFPP)$.
	
Arguing in the same way, $Cop_1 \equiv L^1$ and $Cop_\infty \equiv L^1(1/t)$ (see Example \ref{Cesaro 
i Copson dla L1 i L8} (c)) also fail the weak fixed point property.
	
If $\widetilde{\mathit{X}} \neq \{ 0 \}$, then the claim follows from Proposition \ref{isometric l^infty 
in Tandori} and Remark \ref{izometryczna kopia ell-nieskonczonosc a FPP}. For the space $Ces_\infty$ 
we apply Proposition \ref{Ces_infty isometric copy of l8} and Remark 
\ref{izometryczna kopia ell-nieskonczonosc a FPP}, respectively.
\end{proof}

\section{Generalizations and applications of results}

Until now, most of the results we have obtained for the Ces\`aro and Copson function spaces 
has been proven in the class of Banach function spaces and under the assumption about 
the nontriviality. It turns out that using previously developed methods we can transfer (without 
much effort) the most important results from Sections \ref{sekcja kopii l1} and \ref{sekcja FPP} 
to even more general optimal domains. We start with some definitions.
	
Denote by $\mathcal{H}_w$ the {\it weighted Ces\`aro operator} which is defined as
$$
\mathcal{H}_w \colon f \mapsto \mathcal{H}_w f(x) := w(x) \int_0^x f(t) \mathrm{d}t 
\quad \text{for} \quad t \in I,
$$
where $w$ is a positive weight on $I$. For a Banach function spaces $X$ on $I$ by the 
{\it weighted Ces\`aro function space} $C_{w}X(I) = C_wX$ space we mean
$$ 
C_{w}X := \{ f\in L^0 : \mathcal{H}_w\abs{f} \in X \} \quad \text{with the norm} 
\quad \norm{f}_{C_{w}X} = \norm{\mathcal{H}_w\abs{f}}_X.
$$
These spaces for $X = L^p$, where $1 \leq p < \infty$, were studied by Kami\'nska--Kubiak \cite{KK12} 
and by Kubiak in \cite{Ku14}. Observe, that the study of the spaces $C_{p,w} := \mathcal{H}_w L^p$ 
is more or less equivalent to study of the spaces $CL^p(w)$, that is, the Ces\`aro operator on 
the weighted $L^p$-spaces. Of course, if we take $w(x) = 1/x$ then $C_{w}X \equiv CX$. Moreover, 
if $w \equiv 1$ then $\mathcal{H}_w = V$, where $V$ denote the {\it Volterra operator}
\begin{equation*}
V \colon f\mapsto Vf(x):=\int_{0}^{x}f(t)\mathrm{d}t \quad \text{for} \quad t \in I.
\end{equation*}
Easy computations involving the Fubini's theorem shows that a conjugate operator $\mathcal{H}^*_w$ 
to the weighted Ces\`aro operator $\mathcal{H}_w$ is given by the formula
$$
\mathcal{H}^*_w \colon f \mapsto \mathcal{H}^*_w f(x) := \int_{I \cap [x, \infty)} w(t) f(t) \mathrm{d}t 
\quad \text{for} \quad t \in I.
$$
The space $C_w^*X(I) = C_w^*X$ associated with this operator can be called the {\it weighted 
Copson function space}. Again, if $w(t) = 1/t$ then $C_w^*X \equiv C^*X$ and if $w(t) \equiv 1$ 
then $C_w^*X \equiv V^*X$.
 	
Note that
\begin{equation} \label{C_wX = CY}
\norm{f}_{C_wX} = \norm{ w(x) \int_0^x \abs{f(t)} \mathrm{d}t}_X = \norm{C\abs{f}}_{X(v)} = \norm{f}_{CY},
\end{equation}
that is, $C_wX \equiv CY$, where $Y = X(v)$ and $v(x) := xw(x)$. That is why it is easy to transfer 
claims about the Ces\`aro spaces $CX$ for $X$ being a Banach function space (rather not symmetric) 
to the spaces $C_wX$. Moreover, $C_wX \equiv V(X(w))$.
 	
It is easy to see that the space $C_wX$ is nontrivial if and only if $w(x)\chi_{[\lambda_0, m(I))}(x) \in X$ 
for some $0 < \lambda_0 < m(I)$, cf. (\ref{C_wX = CY}) and \cite[Theorem 1 (a) and (b)]{LM15a}.
Furthermore, if $C_w^*X$ is nontrivial then $\chi_{[0, \lambda_0]} \in X$ for some $0 < \lambda_0 < m(I)$, 
cf. Lemma \ref{Copson nontrivial}. Keeping in mind this observation and following the proofs of 
Lemma \ref{complementedL1} and Lemma \ref{complemented L1 in C*X} we can show that
 	
\begin{lemma} \label{kopie L1 uogolnione C i C*}
{\it Let $X$ be a Banach function space on $I$.
\begin{enumerate}
\item [(i)] Assume that $C_wX(I) \neq \{ 0 \}$. Then there exist $0 \leq \lambda_0 \in I$ with
\begin{equation}
\left\Vert w(x)\chi _{\lbrack b,m(I))}(x)\right\Vert_{X(I)}\left\Vert f\right\Vert _{L^{1}[a,b]}
\leq \left\Vert f\right\Vert_{C_wX(I)}\leq \left\Vert w(x)\chi _{\lbrack a,m(I))}(x)\right\Vert_{X(I)}\left\Vert f\right\Vert _{L^{1}[a,b]},
\end{equation}
for all $f \in C_wX(I)$ such that $\supp(f) \subset [a,b]$, where $0 \leq \lambda_0 < a < b < m(I)$.
\item [(ii)] If $C^*_wX(I) \neq \{ 0 \}$, then we can find $0 < \eta_0 \in I$ with
\begin{equation}
\left\Vert \chi _{[0,a]}\right\Vert_{X(I)}\left\Vert f\right\Vert _{L^{1}(w)[a,b]}
\leq \left\Vert f\right\Vert_{C^*_wX(I)}\leq \left\Vert \chi _{[0,b]}\right\Vert_{X(I)}\left\Vert f\right\Vert _{L^{1}(w)[a,b]},
\end{equation}
for all $f \in C^*_wX(I)$ such that $\supp(f) \subset [a,b]$, where $0 < a < b < \eta_0 \leq m(I)$.
\end{enumerate}
 }
\end{lemma}
 	
The next theorem, in the case of the weighted Ces\`aro function space $C_wX$, easily follows from 
the identification $C_wX \equiv CY$, where $Y = X(v)$ and $v(x) = xw(x)$ (cf. (\ref{C_wX = CY})), 
and Theorem \ref{AIC-l1}. On the other hand, if $T = \mathcal{H}_w^*$, it is sufficient to use the same 
argument as in the proof of Theorem \ref{AIC-l1 in C^*X} and Lemma \ref{kopie L1 uogolnione C i C*} 
(ii) instead of Lemma \ref{complemented L1 in C*X} (actually, the proof will be almost identical, 
because we can use the same function $G_X$). Summarizing the above discussion, we can obtain
 	
\begin{theorem} \label{AIC-l1 uogolnione C i C*}
{\it Let $X$ be a Banach function space on $I$ and $T = \mathcal{H}_w$ or $T = \mathcal{H}_w^*$.
Then the nontrivial space $TX$ contains an order asymptotically isometric copy of $\ell^1$.}
\end{theorem}
 
A similar result for the space $C_{p,w} := C_w L^p$, where $1 \leq p < \infty$, was obtained by 
Kubiak \cite[Theorem 5.1]{Ku14}.

Now, a direct consequence of Theorem \ref{AIC-l1 uogolnione C i C*} and the Dowling--Lenard--Turett 
result \cite{DLT01} (cf. proof of Theorem \ref{FPPCesaro}) is the following

\begin{theorem} \label{uogolnione Ci C* nie maja FPP}
{\it Let $X$ be a Banach function space on $I$ and $T = \mathcal{H}_w$ or $T = \mathcal{H}_w^*$. 
Then the space $TX$ fails the fixed point property whenever it is nontrivial.}
\end{theorem}

As a direct application of the above considerations we can formulate the following

\begin{corollary} \label{VX nie ma FPP}
{\it The nontrivial Volterra space $VX$ contains an asymptotically isometric copy of $\ell^1$ and, 
consequently, fails the fixed point property.}
\end{corollary}

Finally, let us mention also that $Vol_1 := VL^1 \equiv L^1(1-t)$ and $Vol_\infty := VL^\infty \equiv L^1$ 
(cf. Example \ref{Cesaro i Copson dla L1 i L8} (b) and (c)) and consequently
	
\begin{theorem}
{\it The spaces $Vol_1$ and $Vol_\infty$ fail the weak fixed  point property.}
\end{theorem}

\section{Appendix}

In the proof of Theorem \ref{as kopia l1} it was necessary for the function $F_{X}$ to be continuous in at least one point.
It turned out that it is always continuous in uncountably many points. However, the question whether this function is actually
continuous on the whole domain may be of independent interest.

\begin{lemma} \label{Function F}
{\it Let $X$ be a Banach function space on $I$ such that the operator $C$ is bounded on $X$. 
Assume that one of the following holds true
\begin{enumerate}
\item [(i)] the space $X$ is order continuous,
\item [(ii)] the space $X$ is symmetric and $X \not\hookrightarrow L^\infty$.
\end{enumerate}
Then the function $F_{X}$ is finitely valued and continuous for all $0 < x \in I$.}
\end{lemma}

\begin{proof}
Actually, the proof of Lemma \ref{complementedL1} shows that the function $F_X$ is finitely valued.
Moreover, $\supp(CX) = I$ because the Ces\`aro operator is bounded on $X$, see Theorem D.
It remains to prove that $F_X$ is also continuous. We will consider two situations.
	
Assume that $I=[0,1].$ Let us fix $0 < \lambda_{0} < 1$ and take a sequence 
$(\lambda_{n})_{n=1}^{\infty }\subset \lbrack 0,1]$ such that $\lambda_{n}\rightarrow \lambda _{0}$. 
We will show that 
\begin{equation} \label{zb}
F_{X}(\lambda _{n})\rightarrow F_{X}(\lambda _{0}),
\end{equation}
that is, the function $F_X$ is continuous on $(0,1)$. Note first that there exist 
$0 < \varepsilon < \min\{\lambda_0, 1 - \lambda_0 \}$ and $N\in \mathbb{N}$ such that 
\begin{equation} \label{ogr}
0 \leq \abs{f_{\lambda_n}(x) - f_{\lambda_0}(x)} = 
\frac{1}{x}\chi _{(\min \{\lambda_{0},\lambda _{n}\},\max \{\lambda _{0},\lambda _{n}\})}(x) \leq
\max\limits_{n\geq N}\{ \frac{1}{\lambda_0}, \frac{1}{\lambda_n} \} 
\chi _{(\lambda _{0}-\varepsilon ,\lambda _{0}+\varepsilon )}(x),
\end{equation}
for $n\geq N$ and $0<x\leq 1$. Put 
\begin{equation*}
h_{n} := \frac{1}{x}\chi _{(\min \{\lambda_{0},\lambda _{n}\},\max \{\lambda _{0},\lambda _{n}\})} 
\quad \text{and} \quad
H := \max\limits_{n\geq N}\{ \frac{1}{\lambda_0}, \frac{1}{\lambda_n} \} 
\chi _{(\lambda _{0}-\varepsilon ,\lambda _{0}+\varepsilon)}.
\end{equation*}
Then $h_{n}\rightarrow 0$ almost everywhere on $[0,1]$ as $n \rightarrow \infty$, and it follows 
from (\ref{ogr}) that $0 \leq h_{n} \leq H$. We claim that $H\in X_{a}$. In fact, if $X$ is a symmetric 
space and $X \not\hookrightarrow L^\infty$, then $X_a = X_b$, where $X_{b}$ is the closure in 
$X$ of the set of bounded functions supported in sets of finite measure (see, for example, 
Theorem B in \cite{KT17}). It is clear that $H$ is such a function and so $H \in X_{a}$. 
However, if $X$ is an order continuous Banach function space then the situation is a little bit 
different. Due to boundedness of the Ces\`aro operator and Theorem D, we can see that 
$\chi_{[\lambda,1]} \in X$ for all $0 < \lambda < 1$. Therefore, $H \in X = X_a$ and the claim 
follows. Just from the definition of order continuity and (\ref{ogr}) we obtain that
\begin{align*}
0 \leq \left\vert F_{X}(\lambda _{n})-F_{X}(\lambda _{0})\right\vert & =\left\vert
\left\Vert \frac{1}{x}\chi _{(\lambda _{n},1]}(x)\right\Vert _{X}-\left\Vert 
\frac{1}{x}\chi _{(\lambda _{0},1]}(x)\right\Vert _{X}\right\vert \\
& \leq \left\Vert \frac{1}{x}\chi _{(\lambda _{n},1]}(x)-\frac{1}{x}\chi_{(\lambda _{0},1]}(x)\right\Vert _{X} \\
& =\left\Vert \frac{1}{x}\chi _{(\min \{\lambda _{0},\lambda _{n}\},\max
\{\lambda _{0},\lambda _{n}\})}(x)\right\Vert _{X}=\left\Vert
h_{n}\right\Vert _{X}\rightarrow 0\quad \text{as }n\rightarrow \infty .
\end{align*}
This proves (\ref{zb}). In the missing case, when $\lambda_0 = 1$, the argument is essentially 
the same, so we will omit it
	
Now suppose that $I = [0,\infty )$. Note only that in this case $\frac{1}{x}\chi _{\lbrack \lambda ,\infty )}\left( x\right) \in X$ 
for each $0<\lambda \in I$ and $\chi_{[a,b]} \in X$ for each $0 < a < b < \infty$, see Theorem D.
Thus we can proceed as in the previous case.
\end{proof}

It is not surprising that we can prove analogous lemma also for the function $G_X$.

\begin{lemma} \label{ciaglosc G_X}
{\it Let $X$ be a Banach function space on $I$ such that the Copson operator is bounded 
on $X$. Assume that one of the following holds true
\begin{enumerate}
\item [(i)] the space $X$ is order continuous,
\item [(ii)] the space $X$ is symmetric and $X \not\hookrightarrow L^\infty$.
\end{enumerate}
Then the function $G_{X}$ is finitely valued and continuous for all $x \in I$.}
\end{lemma}
\begin{proof}
We proceed as in the proof of Lemma \ref{Function F} and use Corollary \ref{Copson bounded 
and embeddings} instead of the proof of Lemma \ref{complemented L1 in C*X}.
\end{proof}

If we replace the assumption that the operator $T$, where $T = C$ or $T = C^*$, is bounded 
on $X$ by the assumption that $TX \neq \{ 0 \}$, then using the same arguments as before 
we obtain that the function $F_X$ (resp. $G_X$) is finitely valued and continuous for all 
$x \in \interior(\supp(TX))$.

The above lemma does not exclude the possibility that the function $F_X$ is
continuous on $0 < x \in I$ but the space $X$ has trivial order continuous part.
In fact, it is rather common for spaces with $X_a$ being trivial to own this
property. We will now give some examples illustrating the discussion about the continuity
of the function $F_X$.

\begin{example} \label{przyklad do F_X}
(a) Let $X$ be a Banach function space on $I$ and $w_0, w_1 \colon I \rightarrow (0,\infty)$ be 
two weights that differ only on the set of measure zero, i.e., $m(\{x\in I \colon w_0(x) \neq w_1(x) \}) = 0$.
Then, of course, $X(w_0) \equiv X(w_1)$. In particular, if $X$ satisfies the assumptions of 
Lemma \ref{Function F} and $w_1 = \mathfrak{D}$, where
$$
\mathfrak{D} \colon I \ni x \mapsto \mathfrak{D}(x) := \chi_{I/ \mathbb{Q}}(x),
$$
is a Dirichlet function, then $X \equiv X(\mathfrak{D})$ and $\mathfrak{D}$ is nowhere continuous 
function on $I$ but $F_{X(\mathfrak{D})}$ is a continuous function for all $0 < x \in I$.
\smallskip 
	
(b) Put 
$$
w_2 \colon [0,1] \ni x \mapsto w_2(x) := 2\chi_{\mathfrak{C}}(x) + \chi_{[0,1]/\mathfrak{C}}(x),
$$
where $\mathfrak{C}$ is the Smith--Volterra--Cantor set (or the fat Cantor set). The set of 
discontinuities of $w_2$ is the set $\mathfrak{C}$, so it is uncountable and of positive measure.
However, the set of discontinuities of the function $F_Y$, where $Y := L^\infty(w_2)[0,1]$, is 
at most countable.
\smallskip 
	
(c) Let $(q_n) \subset \mathbb{Q} \cap [0,1]$ be a sequence of rational numbers and put 
$Z := L^\infty(w_3)[0,1]$, where
$$
w_3 \colon [0,1] \ni x \mapsto w_3(x) := \sum_{ \substack{q_n < x \\ q_n \in \mathbb{Q} \cap [0,1]}}^\infty 2^{-n}.
$$
Then the function $F_Z$ is discontinuous at every rational number from the interval $[0,1]$ and 
continuous elsewhere.
\smallskip 
	
(d) Let $X$ be a symmetric space such that $X \cong L^{\infty}$, i.e., the spaces $X$ and $L^\infty$ 
have the same elements and $\left\Vert f\right\Vert _{X} = A \left\Vert f\right\Vert _{L^{\infty }}$ for 
$f \in X$ and some constant $A > 0$. Then, with the same notation as in the proof of 
Lemma \ref{Function F}, we have
\begin{equation*}
\left\vert F_{X}(\lambda _{n})-F_{X}(\lambda _{0})\right\vert =\left\vert
\left\Vert \frac{1}{x}\chi _{(\lambda _{n}, m(I))}(x)\right\Vert _{X}-\left\Vert 
\frac{1}{x}\chi _{(\lambda _{0},m(I))}(x)\right\Vert _{X}\right\vert
=A\left\vert \frac{1}{\lambda _{n}}-\frac{1}{\lambda _{0}}\right\vert \rightarrow 0,
\end{equation*}
as $n \rightarrow \infty$. But this means that $F_X$ is continuous for $0 < x \in I$. Is also worth 
noting that if $X$ is a symmetric space on $[0,1]$ then the condition $X_a =\{ 0 \}$ is equivalent 
to $X=L^{\infty}[0,1]$, see \cite[Theorem B]{KT17}. However, in the class of Orlicz spaces 
the condition $(L^\Phi)_a = \{ 0 \}$ (that is, the Orlicz function $\Phi$ takes also infinite values) 
is equivalent to $L^\Phi \cong L^{\infty}[0,1]$. Moreover, for symmetric spaces on $[0, \infty)$ 
the aforementioned condition $X_a =\{ 0 \}$ is equivalent to $X \hookrightarrow L^\infty[0,\infty)$, 
see also \cite[Theorem B]{KT17}.
\smallskip 	
	
(e) Let $Y$ be a Banach function space on $I$ such that $Y \cong X \cap L^\infty$, where
$$
X \cap L^\infty := \{f \in L^0 \colon \norm{f}_Y := \max\{\norm{f}_X, \norm{f}_{L^\infty}\} < \infty \},
$$
and $X$ is an order continuous Banach function space on $I$. Of course, $Y_a = \{ 0 \}$ but we 
can prove that the function $F_{Y}$ is continuous for $0 < x \in I$ (we will give the sketch of 
the proof only for $I = [0,1]$ because the remaining case is the same). Indeed, we have
\begin{align*}
\left\vert F_{Y}(\lambda _{n})-F_{Y}(\lambda _{0})\right\vert & =\left\vert
\left\Vert \frac{1}{x}\chi _{(\lambda _{n},1]}(x)\right\Vert _{Y}-\left\Vert 
\frac{1}{x}\chi _{(\lambda _{0},1]}(x)\right\Vert _{Y}\right\vert \\
& = B \left\vert \left\Vert \frac{1}{x}\chi _{(\lambda _{n},1]}(x)\right\Vert
_{X \cap L^{\infty }}-\left\Vert \frac{1}{x}\chi _{(\lambda_{0},1]}(x)\right\Vert _{X \cap L^{\infty }}\right\vert,
\end{align*}
for some constant $B > 0$. Example \ref{przyklad do F_X} (d) above and Lemma \ref{Function F} 
implies that $F_{L^\infty}$ and $F_X$ are continuous functions for all $0 < x \in I$. Therefore, 
$F_Y = B \max\{F_{X}, F_{L^\infty} \}$ is also continuous for $0 < x \in I$ as a maximum of two 
continuous functions and the claim follows.
\smallskip 	
	
The same examples can be considered also in the context of the function $G_X$.
\end{example}

\end{document}